\documentclass[10pt, reqno]{amsart} 
\usepackage{mathrsfs,esvect}
\usepackage{extarrows}
\usepackage{tikz}
\usetikzlibrary{calc, decorations.pathreplacing, matrix, patterns, arrows.meta}
\usepackage{pgfplots}

\usepackage[margin=1in]{geometry}
\definecolor{deepblue}{RGB}{0,74,155}
\definecolor{crimson}{RGB}{178,34,52}
\usepackage[colorlinks=true, linkcolor=blue, citecolor=red, urlcolor=red, backref=page]{hyperref}

\usepackage{verbatim}
\usepackage{amssymb,amscd,amsfonts,amsbsy}
\usepackage{latexsym}
\usepackage{exscale}
\usepackage{amsmath,amsthm,amsfonts}
\usepackage{mathrsfs}
\usepackage{xcolor} 
\usepackage{esint} 
\usepackage{amssymb} 
\usepackage{stmaryrd}
\usepackage{cite} 
\usepackage{tikz}
\usepackage{enumerate}
\calclayout

\allowdisplaybreaks[3] 

\newtheorem{theorem}{Theorem}[section]
\newtheorem{proposition}[theorem]{Proposition}
\newtheorem{lemma}[theorem]{Lemma}

\newtheorem{definition}[theorem]{Definition}
\newtheorem{remark}[theorem]{Remark}

\makeatletter
\@addtoreset{equation}{section}
\makeatother

\def\Xint1{\mathchoice
	{\XXint\displaystyle\textstyle{1}}%
	{\XXint\textstyle\scriptstyle{1}}%
	{\XXint\scriptstyle\scriptscriptstyle{1}}%
	{\XXint\scriptscriptstyle\scriptscriptstyle{1}}%
	\!\int}
\def\XXint123{{\setbox0=\hbox{$1{23}{\int}$}
		\vcenter{\hbox{$23$}}\kern-.5\wd0}}

\DeclareMathOperator*{\esssup}{ess\,sup}

\newcommand{\aver}[1]{-\hskip-0.46cm\int_{1}}
\newcommand{\textaver}[1]{-\hskip-0.40cm\int_{1}}

\def\supp{\operatorname{supp}}

\def\avint_#1{\mathchoice{\mathop{\kern 0.2em\vrule width 0.6em height 0.69678ex depth -0.58065ex \kern -0.8em \intop}\nolimits_{\kern -0.4em#1}}{\mathop{\kern 0.1em\vrule width 0.5em height 0.69678ex depth -0.60387ex \kern -0.6em \intop}\nolimits_{#1}} {\mathop{\kern 0.1em\vrule width 0.5em height 0.69678ex depth -0.60387ex \kern -0.6em \intop}\nolimits_{#1}} {\mathop{\kern 0.1em\vrule width 0.5em height 0.69678ex depth -0.60387ex \kern -0.6em \intop}\nolimits_{#1}}}

\allowdisplaybreaks

\begin{document}
	\title[Commutators of a class of \(\Psi\)-bounded oscillation operators ...]
	{\bf Commutators of a class of \(\Psi\)-bounded oscillation operators on abstract weighted Orlicz-Morrey spaces endowed with ball-basis and applications}
	

\author[F. Shan]{Fate Shan\textsuperscript{\dag}}
\address{Fate Shan:\\
	School of Mathematics and System Sciences\\
	Xinjiang University\\
	Urumqi 830046 \\
	People's Repunlic of China}\email{107552400558@stu.xju.edu.cn} 
	
\author[W. Wang]{Wei Wang\textsuperscript{\dag}}
\address{Wei Wang:\\
	School of Mathematics and System Sciences\\
	Xinjiang University\\
	Urumqi 830046 \\
	People's Repunlic of China}\email{107552400662@stu.xju.edu.cn}

\author[Y. Lang]{Yuchen Li\textsuperscript{\dag}}
\address{Yuchen Lin:\\
	School of Mathematics and System Sciences\\
	Xinjiang University\\
	Urumqi 830046 \\
	People's Repunlic of China}\email{107552400565@stu.xju.edu.cn}

\author[J. Zhou]{Jiang Zhou*} 
\address{Jiang Zhou:\\
	School of Mathematics and System Sciences\\
	Xinjiang University\\
	Urumqi 830046 \\
	People's Repunlic of China}\email{zhoujiang@xju.edu.cn}

\date{\today}

	
\thanks{The first, second, and third authors were supported by the Research Innovation Program for Postgraduates of Xinjiang Uygur Autonomous Region (No. XJ2026G070). The fourth author gratefully acknowledges support from the National Natural Science Foundation of China (No. 12461021). \\
\textsuperscript{\dag} These authors contributed equally to this work and are co-first authors. \\
$^{*}$ Corresponding author, Email: zhoujiang@xju.edu.cn. }

	\begin{abstract}
		In our previous work \cite{SZ2026}, we established a class of abstract Orlicz-Morrey spaces endowed with a ball-basis and introduced the notion of \(\Psi\)-bounded oscillation operators. The present paper further develops this theory by investigating sparse domination and weighted estimates for such operators and their commutators within the framework of abstract Orlicz-Morrey spaces endowed with a ball-basis. We obtain pointwise sparse domination results for the operators and establish their boundedness on weighted Orlicz-Morrey spaces. The hypotheses of these conclusions differ from those in the classical Orlicz-Morrey setting and do not rely on dyadic decompositions of Euclidean structure. Moreover, we study the duality of these spaces and prove that the dual of the associated block space is exactly the Orlicz-Morrey space under consideration.
	\end{abstract}

	\maketitle
	\tableofcontents

	\section{Introduction}\label{Introduction}\label{Section1}
	
	\subsection{Some reflections from measure spaces to abstract spaces endowed with a ball-basis}\label{Section1.1}~
	
	
	The study of operators acting on classical function spaces remains an active field of investigation. Scholars often adopt two main approaches in this context: one is the boundedness of operators, and the other is the compactness of operators. The boundedness can be further divided into diagonal-type estimates and off-diagonal-type estimates. In our previous work \cite{SZ2026}, we further investigated the estimates for \(\Psi\)-BOOs (Definition \ref{def-BOO}) and their commutators on abstract Orlicz-Morrey spaces endowed with a ball-basis, and also discussed the dual spaces of such spaces. Below, we elaborate on these topics step by step in line with the content of this paper.

	{\bf (1) Classical theory of Orlicz-Morrey spaces}
	
	
	Based on our review of the classical Orlicz-Morrey spaces, it is now known that there exist at least five types of such spaces defined on the underlying space \((\mathbb{R}^n, \mathscr{L}^n)\) (where \(\mathscr{L}^n\) denotes the standard Lebesgue measure on \(\mathbb{R}^n\)): the KK-type, Nakia-type, SST-type, DGS-type, and Ho-type (see Appendix \ref{Appendix5} for their explicit forms and related results). Below we describe the background of the Orlicz-Morrey spaces employed in our work \cite{SZ2026}.


	The series of works by Nakai \cite{OM-Nakai-2004, OM-Nakai-2006, OM-Nakai2008, OM-Nakai-2011} laid the foundation for the development of Nakia-type Orlicz-Morrey spaces. The Orlicz-Morrey spaces defined in \cite{OM-Nakai-2004} (see \eqref{Nakai2004}) serve as a common extension covering \(L^p(\mathbb{R}^n,\mathscr{L}^n)\), Morrey, and Orlicz spaces. In that same work, the author showed that the Hardy-Littlewood maximal operator is bounded on these spaces. A subsequent paper \cite{OM-Nakai-2006} further proved the boundedness of Calder\'on-Zygmund operators and also derived modular inequalities for Orlicz spaces. In \cite{OM-Nakai2008, OM-Nakai-2011}, Nakai further verified that these spaces contain \(L^p (\mathbb{R}^n, \mathscr{L}^n)\), \(L^\infty (\mathbb{R}^n, \mathscr{L}^n)\), (generalized) Morrey, Orlicz, and other classical spaces, defined weak Orlicz-Morrey spaces, proved H\"older’s inequality, and discussed the dual spaces and inclusion relations among different Orlicz-Morrey spaces. The question of when the commutators \([b,T]\) and \([b,I_{\rho}]\) are bounded on Orlicz-Morrey spaces was settled by Shi et al. \cite{OM-Nakai-2021-1}, under the assumptions that \(T\) is a Calder\'on-Zygmund operator, \(I_{\rho}\) a generalized fractional integral operator, and \(b\) belongs to the Campanato class. Later, Yamaguchi et al. \cite{OM-Nakai-2022} studied the compactness properties of these same commutators. For further related results, we refer the reader to \cite{OM-Nakai-2019, OM-Nakai-2023, OM-Nakai-2021-1, OM-Nakai-2021-2, OM-Nakai-2024}.

	
	It is a classical fact that Orlicz spaces and Morrey spaces were originally devised as generalizations of Lebesgue spaces. The former were introduced by Morrey \cite{Morrey1938} in the context of second-order elliptic PDEs, where he observed that many qualitative properties of their solutions can be reformulated in terms of operator boundedness on Morrey spaces. These spaces have also proved useful in describing the exact behaviour of Riesz potentials \cite{Adams1975} and fractional integral operators \cite{Olsen1995}. In contrast, the Hardy-Littlewood maximal operator, though bounded on every \(L^p(\mathbb{R}^n,\mathscr{L}^n)\) with \(1 < p < \infty\), fails to be bounded on \(L^1 (\mathbb{R}^n, \mathscr{L}^n)\); Orlicz spaces offer a natural environment for studying its boundedness near the critical endpoint \(p = 1\) \cite{Kita1996, Cianchi1999}. By merging Orlicz and Morrey spaces into a single framework, Orlicz-Morrey spaces are of broader theoretical relevance, and both constituent classes have been the subject of extensive research \cite{Morrey1964, Wmq2024, Peetre1966, Peetre1969, Nakai2000, OM-Nakai2008, OM-SST-2012, OM-DGS-2014, Xue2025, MorreyBook1, MorreyBook2}. Consequently, the development of a solid abstract theory for Orlicz-Morrey spaces is expected to substantially widen the applicability of the classical theory.

	{\bf (2) Sparse control techniques}
	
	
	In harmonic analysis, operators can be roughly divided into Calder\'on-Zygmund operators and those lying beyond the Calder\'on-Zygmund theory, depending on whether they satisfy the classical Calder\'on-Zygmund kernel conditions and possess weak type \((1, 1)\) boundedness as well as \(L^p (\mathbb{R}^n, \mathscr{L}^n)\) boundedness. Examples of the latter include singular integrals with rough kernels, oscillatory singular integrals, Fourier integral operators, multilinear singular integrals, and operators on non-doubling measure spaces. For such operators, pointwise control and norm estimates are commonly employed to describe their behavior, with the fundamental reason being the crucial role played by their oscillatory nature.


	The theory of sparse control originates from the seminal work of Lerner. In \cite{Lerner2010}, Lerner introduced the technique of local mean oscillation, and in \cite{Lerner2013, Lerner2016}, he defined sparse collections and sparse operators. Subsequently, Karagulyan introduced two pairwise disjoint \(\beta\)-sparse families (cf. Definition \ref{sparse}) on a measure space \((X, \mathfrak{M}, \mu)\) endowed with a ball-basis, and constructed the sparse averaging operator \(\mathcal{A}_{\mathcal{S},r}\) (where \(\mathcal{S}\) is the union of the two sparse families). He further established the following sparse domination for the bounded oscillation operator \(T_K\):
	\begin{align}\label{K-sparse}
		\left| T_K f(x) \right| \lesssim \left( \mathcal{L}_1 (T_K) + \mathcal{L}_2 (T_K) +\|T_K\|_{L^r \to L^{r, \infty}} \right) \cdot \mathcal{A}_{\mathcal{S}, r} f(x) \quad a.e. \; x \in B.
	\end{align}
	We refer to the estimate above as the {\tt K-pattern sparse control}; see \cite[Theorem 1.1]{Abstract2019} for details.
	It is noteworthy that Lerner also pointed out in \cite[Theorem 4.2]{Lerner2016} that both the sparse domination of an operator \(T\) and the estimate \eqref{K-sparse} imply that an operator must satisfy a certain weak boundedness condition in order to be pointwise dominated by a sparse operator. Following this direction, Lerner introduced in \cite{Lerner2020} the oscillation operator
	\begin{align*}
		\mathcal{M}_{T,\alpha}^{\#}f(x) := \sup_{Q \ni x} \; \esssup_{x',x'' \in Q} \left| T(f\mathbb{I}_{\mathbb{R}^n \setminus \alpha Q})(x') - T(f\mathbb{I}_{\mathbb{R}^n \setminus \alpha Q})(x'') \right|,
	\end{align*}
	for a sublinear operator \(T\), as well as the \(W_q\) condition for a weak type \((q, q)\) operator \(T\):
	\begin{align*}
		\left| \left\{ x\in Q: |T(f \, \mathbb{I}_Q)(x)| > \psi_{T,q} (\lambda) \left( \fint_Q |f|^q \right)^q \right\} \right| \leq \lambda|Q|, \, 0 < \lambda < 1.
	\end{align*}
	Lerner established that the \(W_q\) condition together with a weak boundedness property of \(\mathcal{M}_{T, \alpha}^{\#}\) provides nearly minimal assumptions under which an operator can be pointwise controlled by a sparse operator.
	


	Not only classical Calder\'on-Zygmund operators, but also many important operators beyond the Calder\'on-Zygmund class, can be dominated by sparse operators--either pointwise or in norm--as shown by the development of sparse control theory. The essential reason for this unifying phenomenon lies precisely in the establishment and deepening of the sparse control theory. Sharp weighted estimates, including optimal dependence on the \(A_p\) characteristic, are known for maximal and sparse operators. Hence, whenever an operator is sparse-dominated, its weighted norm bounds can be derived from the sparse ones, thereby giving sharp--and in some cases optimal--constants. This provides a direct and powerful tool for quantitative weighted estimates of various operators in harmonic analysis.


	Building on this foundation, the theory of sparse control has seen further advances. Nearly minimal assumptions for pointwise sparse domination of oscillation operators were identified by Lerner in \cite[Theorem 1.1]{Lerner2020}. Cao, in \cite[Theorem 1.5]{Caomingming2023}, refined the parameter \(\beta\) in \eqref{K-sparse} to the exact value \(\beta = \left( 2\mathscr{C}_0^3 \right)^{-1}\). The Anh Bui \cite{sparse2026} employed sparse control techniques to study weighted estimates for a class of fractional operators lying beyond Calder\'on-Zygmund theory. In \cite[Theorem 1]{Shan2026}, we obtained an application of \eqref{K-sparse} to abstract Morrey spaces endowed with a ball-basis. Furthermore, sparse control has also led to important progress in the study of commutators, Bloom weighted estimates, and other topics \cite{sparse2021, Lerner2024, XQYsparse2025, CYPsparse2025}. In modern analysis, these approaches now form a foundational component of quantitative estimation theory, and they keep advancing research into boundedness of operators across diverse function spaces \cite{LKW2018, Lerner2024, sparse2026, OM-sparse2025, CYPsparse2025, XQYsparse2025, Lerner2019, Xue2025}.

	(3) {\bf Abstract function space theory endowed with a ball-basis}
	
	
	A major factor behind the fast growth of dyadic analysis is Hyt\"onen’s refinement of the random dyadic system approach; this refinement led directly to the resolution of the \(A_2\) conjecture in the general Calder\'on-Zygmund setting \cite{Hytonen2012}. Subsequently, Lerner’s sparse control theorem established that every dyadic shift can be pointwise controlled by a sparse operator, thereby reducing complex random averages to clean weighted estimates on sparse families. Lacey \cite{Lacey2017} further streamlined this approach: he showed that any \(\theta\)-Calder\'on-Zygmund operator satisfying the Dini condition, as well as martingale transforms, admits pointwise sparse control, bypassing the random averages and parameter summations inherent in dyadic shift representations, and relying only on weak type \((1, 1)\) estimates and recursive covering to derive the optimal \(A_2\) bound. Following this line, Di Plinio et al. \cite{Di Plinio2014} obtained sharp weighted estimates for the Carleson operator.


	However, the aforementioned works often require different technical tools for different operators, lacking a unified framework. To address this shortcoming, Karagulyan incorporated these operators into a comprehensive abstract theory: he introduced the notion of a ball-basis on a measure space \((X, \mathfrak{M}, \mu)\), which unifies classical metric balls and dyadic cubes, and defined bounded oscillation operators. This framework encompasses the Hardy-Littlewood maximal operator on \((\mathbb{R}^n, \mathscr{L}^n)\) \cite{Muckenhoupt1972, Buckley1993}, Calder\'on-Zygmund operators (including \(\theta\)-Calder\'on-Zygmund operators) on \((\mathbb{R}^n, \mathscr{L}^n)\) and on spaces of homogeneous type \cite{Coifmany1974, Dragicevic2005, Hytonen2012, Lacey2017, H-space}, martingale transforms \cite{Izumisawa1977, Lacey2017, Thiele2015}, as well as Carleson-type maximally modulated operators on \((\mathbb{R}^n, \mathscr{L}^n)\) \cite{Grafakos2005, Di Plinio2014}. Karagulyan proved that bounded oscillation operators can be pointwise dominated by sparse operators, and consequently established the following weighted estimate:
	\begin{align}\label{K-norm}
		\|T\|_{L^p(\omega) \to L^p(\omega)} \leq c(p,\mathscr{C}_0) \left( \mathcal{L}_1(T) + \mathcal{L}_2(T) +\|T\|_{L^1 \to L^{1,\infty}}\right) [\omega]_{A_p}^{\max\left\{\frac{p+2}{p(p-1)},\frac{3p-2}{p}\right\}},
	\end{align}
	where \(\mathscr{C}_0\) is the constant appearing in condition \eqref{list:BB4}. We refer to the above estimate as the {\tt K-pattern weighted estimate}, and call a space possessing this property an {\tt abstract function space endowed with a ball-basis}.


	The impact of this groundbreaking work was substantial, prompting a wave of research into abstract function spaces with a ball-basis. Cao \cite{Caomingming2023} later improved the ball-basis definition given in \cite{Abstract2019} and introduced a new category of Banach-valued multilinear bounded oscillation operators on abstract Lebesgue spaces. According to his results, this category covers both multilinear Calder\'on-Zygmund-type operators, including the multilinear Hardy-Littlewood maximal operator and multilinear \(\theta\)-Calder\'on-Zygmund operators, and those beyond the multilinear Calder\'on-Zygmund framework, such as multilinear Littlewood-Paley square functions and multilinear Carleson-type operators. For these operators, he obtained both K-pattern sparse control and K-pattern weighted estimates for the multilinear bounded oscillation operators and their commutators. Karagulyan \cite{Abstract2023} extended the previously developed theory of bounded oscillation operators to vector-valued function spaces, introduced more general parameters, thereby incorporating a larger family of classical operators into a unified framework, and proved that such operators admit sparse control. For further results within this framework, see \cite{Cenxi}.
	
	
	Furthermore, the universality of this abstract framework has enabled its rapid extension to various concrete function spaces. For instance, abstract variable-exponent Lebesgue spaces endowed with a ball-basis \cite{Wangsb}, abstract generalized Orlicz spaces endowed with a ball-basis \cite{Zhou2025}, abstract Morrey spaces endowed with a ball-basis \cite{Shan2026} and abstract weighted Morrey spaces endowed with a ball-basis \cite{Liyc}, as well as abstract BMO \cite{Abstract2026-BMO-BLO, Abstract2026-BMO} and BLO spaces \cite{Abstract2026-BMO-BLO} endowed with a ball-basis have all been incorporated into this theoretical framework, with corresponding sparse control and weighted estimates established. These works further demonstrate the strong adaptability of the ball-basis framework when handling singular integral operators under different structural settings.


	In summary, the theory of bounded oscillation operators and abstract function spaces endowed with a ball-basis, as proposed by Karagulyan, provides a unified processing framework for Calder\'on-Zygmund operators, maximal functions, martingale transforms, and Carleson-type operators in harmonic analysis. This framework has been extended to vector-valued, multilinear, and various function spaces with variable exponents. It has successfully yielded sparse control and weighted estimates for the aforementioned operators, encompassing the major developments from dyadic analysis to the resolution of the \(A_2\) conjecture.

	\subsection{Main results and related contributions}\label{Section1.2}~~


	In this paper, we obtain the following principal results. The first one provides a pointwise domination of \(\Psi\)-BOOs by means of sparse operators.

	\begin{theorem}\label{main}
		Let \((X, \mathfrak{M}, \mu)\) be a measure space endowed with a ball-basis \(\mathfrak{B}\), let \(b\) be a measurable function, and suppose \(\lambda > 3 \mathscr{C}_0^6\). Assume that \((\Phi, \Psi, \phi, \psi) \in \mathscr{Y} \otimes \mathscr{G}\) and \((\Phi, \phi) \in E_{mb}\). Let \(T\) be a \(\mathbb{B}\)-valued \(\Psi\)-BOO satisfying \(T \in \mathbb{W}_{\Psi, \psi, \lambda}\), and let \(\mathscr{T}\) be a \(\mathbb{B}\)-valued linear operator such that \(Tf(x) = \|\mathscr{T}f(x)\|_\mathbb{B}\). Then there exist two families \(\mathcal{S}_1, \, \mathcal{S}_2 \subset \mathfrak{B}\), each of which is \(\frac{1}{2 \mathscr{C}_0^3}\)-sparse, such that for every \(f \in L^{(\Psi,\psi)}(X) \subset L^{(\Phi,\phi)}(X)\) and every \(B \in \mathfrak{B}\),
		\begin{align}\label{ine main1}
			\|\mathscr{T} f(x)\|_\mathbb{B} \lesssim \mathscr{C} (T) \left[ \mathcal{A}_{\mathcal{S}_1, \Phi,\phi} f(x) + \mathcal{A}_{\mathcal{S}_2, \Phi,\phi} f(x) \right], \quad a.e. \; x \in B,
		\end{align}
		and
		\begin{align}\label{ine main2}
			\|[\mathscr{T}, b] f(x)\|_\mathbb{B} \lesssim \mathscr{C} (T) \left[ \mathcal{A}^b_{\mathcal{S}_1, \Phi,\phi} f(x) + \mathcal{A}^b_{\mathcal{S}_2, \Phi,\phi} f(x) \right], \quad a.e. \; x \in B,
		\end{align}
		where \(\mathscr{C} (T) := \mathscr{C}_1 (T) + \mathscr{C}_2 (T) + \|T\|_{L^{(\Psi, \psi)}(X) \to wL^{(\Psi, \psi)}(X)}\). Here \(\mathscr{C}_0\) is the constant from Definition \ref{def-ball-basis}, while \(\mathscr{C}_1 (T)\) and \(\mathscr{C}_2 (T)\) are the constants appearing in Definition \ref{def-BOO}.
	\end{theorem}


	The conditions \((\Phi, \Psi, \phi, \psi) \in \mathscr{Y} \otimes \mathscr{G}\), \((\Phi, \phi) \in E_{mb}\), and \(T \in \mathbb{W}_{\Psi, \psi, \lambda}\) appearing in Theorem \ref{main} are introduced in Definition \ref{Y-G condition}, Definition \ref{Emb condition}, and Definition \ref{Assume condition}, respectively. Our second main result concerns the norm estimate of \(\Psi\)-BOO operators on abstract Orlicz-Morrey spaces endowed with a ball-basis.
	
	\begin{theorem}\label{main'}
		Let \((X, \mathfrak{M}, \mu)\) be a measure space endowed with a ball-basis \(\mathfrak{B}\) and let \(\lambda > 3 \mathscr{C}_0^6\). Suppose \(\omega \in A^\mathfrak{B}_p\). Assume that \(\mathfrak{B}\) fulfills the Besicovitch \(\mathfrak{N}\)-condition, that \((\Phi, \Psi, \phi, \psi)\in\mathscr{Y} \otimes \mathscr{G}\), and that \((\Phi, \phi) \in E_{mb}\). If \(T\) is a \(\mathbb{B}\)-valued \(\Psi\)-BOO and \(T\) belongs to \(\mathbb{W}_{\Psi, \psi, \lambda}\), then the following hold:
		\begin{align}\label{ine main'1}
			\|T\|_{L^{(\Psi, \psi)}(\omega) \to L^{(\Psi,\psi)}(\omega)} \lesssim \mathscr{C}(T) \, \mathscr{K} \, \mathfrak{N}.
		\end{align}
		and
		\begin{align}\label{ine main'2}
			\|[T, b]\|_{L^{(\Psi, \psi)}(\omega) \to L^{(\Psi,\psi)}(\omega)} \lesssim \mathscr{C}(T) \, \mathscr{K} \, \mathfrak{N} \, \|b\|_{\mathfrak{L}^1_\psi(\omega)}.
		\end{align}
		Here \(\mathfrak{N}\) and \(\mathscr{K}\) denote the constants appearing in \eqref{N} and \eqref{K}, respectively, and \(\|\cdot\|_{\mathfrak{L}^1_\psi(\omega)}\) is as defined in \eqref{Campanato norm}.
	\end{theorem}

	
	Finally, we discuss the properties of abstract Orlicz-Morrey spaces endowed with a ball-basis, and obtain the dual space of such a space.
	
	\begin{theorem}\label{main2}
		Let \((X, \mathfrak{M}, \mu)\) be a measure space endowed with a ball-basis \(\mathfrak{B}\), and let \(\Psi \in \mathcal{Y}\) and \(\phi \in \mathcal{G}\). Then
		\begin{align*}
			\left( \mathscr{B}_{(\Phi, \phi)}(X) \right)^* = L^{(\widetilde{\Phi}, \phi)} (X).
		\end{align*}
		where \(\mathscr{B}_{(\Phi, \phi)}(X)\) is as defined in \eqref{B-space}.
	\end{theorem}

	We highlight the core contributions of this work as below:
	\begin{itemize}
		
		\item Our established framework of Banach-valued \(\Psi\)-BOOs integrates the Orlicz-Morrey maximal operator, \(\theta\)-Calder\'on-Zygmund operators, and those operators lying beyond the Calder\'on-Zygmund theory. It should be noted that we incorporate the Orlicz-Morrey maximal operator into the theory of singular integrals, and we classify the intrinsic square operators as well as Carleson-type operators as belonging to the realm beyond Calder\'on-Zygmund theory.


		\item The abstract weighted Orlicz-Morrey spaces endowed with a ball-basis that we have constructed encompass a wide variety of abstract spaces, including abstract Orlicz-Morrey spaces \cite{SZ2026}, abstract Orlicz spaces \cite{Zhou2025}, abstract Morrey spaces \cite{Shan2026}, and abstract weighted Morrey spaces \cite{Liyc}. They also cover classical weighted Orlicz-Morrey spaces \cite{weight}, central Orlicz-Morrey spaces \cite{COM-2015}, Orlicz-Morrey spaces on spaces of homogeneous type \cite{OM-DGS-2019}, and local Orlicz-Morrey spaces \cite{LOM}. Given the generality of the Orlicz-Morrey framework, it can also incorporate other classical function spaces, such as those of Orlicz type and Morrey type.


		\item Our first result, Theorem \ref{main}, establishes a pointwise characterization of \(\Psi\)-BOOs and their commutators within abstract Orlicz-Morrey spaces endowed with a ball-basis; this serves as a local description. Using Theorem \ref{main}, we obtain our second result, Theorem \ref{main'}, in the setting of weighted Orlicz-Morrey spaces endowed with a ball-basis. Our approach differs from classical proofs \cite{OM-Nakai-2004, OM-Nakai-2006, OM-Nakai2008, OM-DGS-2015, OM-Nakai-2021-1}, where the key obstacle is the absence of dyadic structure in general measure spaces—a feature that Euclidean spaces crucially rely on to make classical arguments work. Our third result, Theorem \ref{main2}, investigates the dual space of abstract Orlicz-Morrey spaces endowed with a ball-basis. Together, these results constitute a foundational framework for operator theory and duality on abstract Orlicz-Morrey spaces endowed with a ball-basis.


		\item As applications, in Sect. \ref{Section2} we present several common operators in Orlicz-Morrey spaces: the Orlicz-Morrey maximal operator, \(\theta\)-Calder\'on-Zygmund operators, intrinsic square operators, and Carleson-type operators. Making use of Theorem \ref{main} and Theorem \ref{main'}, we establish pointwise estimates and weighted norm inequalities for these classical operators and their commutators (see Theorem \ref{example-CZ}, Theorem \ref{example-intrinsic}, and Theorem \ref{example-Carleson}). The hypotheses we adopt are not the same as those in \cite{OM-Nakai-2004, OM-Nakai-2006, OM-Nakai2008, OM-DGS-2015, OM-Nakai-2021-1}.

	\end{itemize}

	\subsection{Basic assumptions}\label{Section1.3}~
	

	The spaces we investigate are broader than Euclidean ones, even though Euclidean spaces already represent a generalization of classical distance. Two essential tasks emerge: developing a more general metric and substituting standard geometric figures (like balls and cubes) with adequate counterparts. This process leads directly to the definition of a ball-basis \cite{Abstract2019}, which is later used to build abstract function spaces equipped with this object.

	\begin{definition}\label{def-ball-basis}
		Let \((X, \mathfrak{M}, \mu)\) be a measure space. A collection \(\mathfrak{B} \subset \mathfrak{M}\) is called a {\tt ball-basis} if it fulfills the following properties:
		
		\begin{list}{\rm (\theenumi)}{\usecounter{enumi}\leftmargin=1.2cm \labelwidth=1cm \itemsep=0.2cm \topsep=.2cm \renewcommand{\theenumi}{BB\arabic{enumi}}}
			\item\label{list:BB1} Every member \(B \in \mathfrak{B}\) satisfies \(0 < \mu(B) < \infty\).
			
			\item\label{list:BB2} For any two points \(x, y \in X\), there exists some \(B \in \mathfrak{B}\) containing both \(x\) and \(y\).
			
			\item\label{list:BB3} For every measurable set \(E \in \mathfrak{M}\) and any \(\varepsilon > 0\), one can select a finite or infinite sequence of balls \(\{B_k\} \subset \mathfrak{B}\) (with \(k = 1, 2, \dots\)) such that \(\mu\left(E \triangle \bigcup_k B_k\right) < \varepsilon\).
			
			\item\label{list:BB4} For each \(B \in \mathfrak{B}\), there exists a distinguished ball \(B^\P \in \mathfrak{B}\), called the {\tt ball-hull} of \(B\), satisfying:
			\[
			\bigcup_{\substack{B'\in\mathfrak{B}:\, B'\cap B \neq \varnothing \\ \mu(B') \le 2\mu(B)}} B' \;\subseteq\; B^\P, \qquad 
			A \subset B \;\Longrightarrow\; A^\P \subset B^\P,
			\]
			and the measure estimate
			\[
			\mu(B^\P) \le \mathscr{C}_0 \, \mu(B),
			\]
			where \(\mathscr{C}_0 \ge 1\) is a fixed constant.
		\end{list}
		
		In this case, we say that \((X, \mathfrak{M}, \mu)\) is a {\t measure space endowed with a ball-basis}.
	\end{definition}

	\begin{remark}\label{example-ball-basis}
		The following examples are some common geometric structures, all of which are ball-basis.
		
		\begin{list}{\rm (\theenumi)}{\usecounter{enumi}\leftmargin=1.2cm \labelwidth=1cm \itemsep=0.2cm \topsep=.2cm \renewcommand{\theenumi}{Ex\arabic{enumi}}}
			\item\label{list:R1} \textbf{Euclidean balls}. In \(\mathbb{R}^n\), consider the family of Euclidean balls
			\[
			\mathfrak{B}_E = \bigl\{ B(x_k, r) : x_k \in \mathbb{R}^n,\ k=1,2,\dots,\ r>0 \bigr\},
			\]
			where \(B(x_k, r) = \{ x \in \mathbb{R}^n : \|x_k - x\|_{l^2} \le r \}\). Then \(\mathfrak{B}_E\) forms a ball-basis on \((\mathbb{R}^n, \mathscr{L}^n)\). Here \(\|x-y\|_{l^2} = \left(\sum_{i=1}^n (x_i-y_i)^2\right)^{1/2}\). In fact, \eqref{list:BB1}, \eqref{list:BB2} and \eqref{list:BB3} are easily verified. For any fixed \(\ell \geq 1+2^{1 + 1/n}\), taking \(\mathscr{C}_0 = \ell^n\) and setting \(B^\P = \ell B\) for every \(B \in \mathfrak{B}_E\) yields \eqref{list:BB4}.

			\item\label{list:R2} \textbf{Dyadic lattices}. On \((\mathbb{R}^n, \mathscr{L}^n)\), any dyadic lattice \(\mathfrak{D}\) is a ball-basis. Recall that a dyadic lattice \(\mathfrak{D}\) in \(\mathbb{R}^n\) is a collection of cubes enjoying the following properties:
			
			\begin{itemize}
				\item \textit{Closure under dyadic descendants}: If a cube \(Q\) belongs to \(\mathfrak{D}\), then all cubes obtained by repeatedly halving its sides (its dyadic children) also lie in \(\mathfrak{D}\). Denoting by \(D(Q)\) the set of all descendants of \(Q\), we have \(D(Q)\subset\mathfrak{D}\).
				
				\item \textit{Existence of a common ancestor}: For any two cubes \(Q_1, Q_2 \in \mathfrak{D}\), there exists some \(Q \in \mathfrak{D}\) such that both \(Q_1\) and \(Q_2\) are dyadic descendants of \(Q\); that is, \(Q_1, Q_2 \in D(Q)\).
				
				\item \textit{Covering of compact sets}: Every compact subset \(K\subset\mathbb{R}^n\) is contained in some cube \(Q \in \mathfrak{D}\).
			\end{itemize}

			\item \textbf{\(\rho\)-ball}: In a homogeneous type space \((X,\,\rho,\, \mathfrak{M},\,\mu)\), we consider the family of balls \(B(x,\,r) = \left\{y \in X : \rho(x , y) < r,\, r>0 \right\}\) denoted by \(\mathfrak{B}_{\rho}= \left\{B(x,\,r): x \in X\right\}\subset \mathfrak{M}\). Then \(\mathfrak{B}'_{\rho}\) is a ball-basis \cite[Theorem 7.1]{Abstract2019}, where
			\[
			\mathfrak{B}'_{\rho}= 
			\begin{cases}
				\mathfrak{B}_{\rho}\cup\{X\}, & \mu(X)<\infty, \\[4pt]
				\mathfrak{B}_{\rho}, & \mu(X)=\infty.
			\end{cases}
			\]

			\item \textbf{Martingale Basis}: Consider a measure space \((X,\,\mathfrak{M},\,\mu)\) endowed with a martingale basis \(\mathfrak{B}\). This structure is characterized by the following conditions: 
			\begin{itemize}
				\item[\(\bullet\)] The collection \(\mathfrak{B}=\bigcup_{j \in \mathbb{Z}} \mathfrak{B}_j\) is a basis for the \(\sigma\)-algebra of measurable subsets of  \(X\).
				\item[\(\bullet\)] For every integer \(j\), the collection \(\mathfrak{B}_j\) constitutes a finite or countable partition of \(X\).  
				\item[\(\bullet\)]  For every integer \(j\), any set \(B \in \mathfrak{B}_j\) can be expressed as a union of sets from \(\mathfrak{B}_{j+1}\). 
				\item[\(\bullet\)] For any two points \(x,\, y \in X\), one can find a set \(B \in \mathfrak{B}\) that contains both \(x\) and \(y\).
			\end{itemize}
			For each set \(B\) in the basis \(\mathfrak{B}\), let \(\mathscr{F}(B)\) be the minimal set in \(\mathfrak{B}\) that contains \(B\). Define \(\mathscr{F}^1 = \mathscr{F}\) and, for \(k \ge 1\), set \(\mathscr{F}^{k+1} = \mathscr{F}(\mathscr{F}^k)\). The associated expanded set \(B^\P\) is then given by
			\[B^\P= 
			\left\{ 
			\begin{array}{ll}
				B, & \mu(\mathscr{F}(B)) > 2\mu(B), \\
				\mathscr{F}^k(B),  & \mu(\mathscr{F}^k(B)) \leq 2\mu(B) < \mu(\mathscr{F}^{k+1}(B)). \\
			\end{array}
			\right.\]
			It is straightforward to verify that \(\mathfrak{B}\) forms a ball-basis for the measure space \((X,\,\mathfrak{M},\,\mu)\) with constant \(\mathscr{C}_0 = 2\). Notably, the measure \(\mu\) in this setting may be non-doubling.
			
		\end{list}
		
		However, it should be noted that the following example does not meet the requirements of a ball-basis.

		\begin{list}{\rm (\theenumi)}{\usecounter{enumi}\leftmargin=1.2cm \labelwidth=1cm \itemsep=0.2cm \topsep=.2cm \renewcommand{\theenumi}{R\arabic{enumi}}}
			
			\item \textbf{Standard dyadic cubes}: Take the standard dyadic cubes in \(\mathbb{R}^n\), whose collection is denoted by \(\mathcal{D} = \bigcup_{k\in\mathbb{Z}}\mathcal{D}_k\), with each level given by
			\[
			\mathcal{D}_k=\left\{Q:\,Q=\prod_{i=1}^n[2^km_i,\,2^k(m_i+1)),\,k\in\mathbb{Z},\,m_i\in\mathbb{Z},\,i=1,2,\ldots,n\right\}.
			\]
			Evidently, \(\mathcal{D}\) fails to meet condition \eqref{list:BB2}. Consequently, it cannot serve as a ball-basis on \((\mathbb{R}^n, \mathscr{L}^n)\).

			\item \textbf{Parallel rectangles}: It is not true that the family of all rectangles with sides parallel to the axes forms a ball-basis on \((\mathbb{R}^n,\mathscr{L}^n)\). For instance, in the plane, set
			\[B = [0,\,1) \times [0,\,1)~~~\text{and}~~~B_k = [0,\,2^{k+1}) \times [0,\,2^{-k}),\,k = 0,1,2,....\]
			For every \(k\), \(B_k\cap B\ne\varnothing\) and \(|B_k|=2|B|\). Yet the union \(\bigcup_k B_k\) cannot be contained in any axis-parallel rectangle. This shows that condition \eqref{list:BB4} fails.

			\item \textbf{Zygmund rectangles}: In the space \((\mathbb{R}^3,\mathscr{L}^3)\), take the family \(\mathfrak{B}\) consisting of all Zygmund rectangles—that is, axis-parallel rectangles whose side lengths are proportional to \(s\), \(t\), and \(st\) for some \(s,t>0\). This family does not constitute a ball-basis. A counterexample is obtained by setting
			\[B = [0,1)^3,~~~\text{and}~~~B_k = [0, 2^{k+1/2}) \times [0, 2^{-k}) \times [0, 2^{1/2}),\,k = 0,1,2,....\]
			Each \(B_k \in \mathfrak{B}\) and satisfies \(B_k \cap B \neq \varnothing\) and \(|B_k| = 2|B|\). However, no Zygmund rectangle in \(\mathfrak{B}\) can contain the union \(\bigcup_{k=0}^\infty B_k\). Hence condition \eqref{list:BB4} of a ball-basis is violated.

			\item \textbf{Centrally shrinking cube sequence}: Consider the subset of \(\mathbb{R}^n\) given by the cube \[Q = [a, b]^n = \{(x_1, \dots, x_n) : x_i \in [a, b],\, i=1,2,\dots, n \},\] where \(a < b\), and the corresponding space \((Q, \mathscr{L}^n)\). The centrally shrinking cube sequences are defined as
			\[Q_k=\left[a+\sum_{l=1}^k\frac{b-a}{2^{l+1}},b-\sum_{l=1}^k\frac{b-a}{2^{l+1}}\right]^n,\,\,\widetilde{Q}_k=\left[a+\frac{b-a}{2^{k+1}},b-\frac{b-a}{2^{k+1}}\right]^n,\quad k=1,\,2,\,....\]
			The family \(\mathfrak{Q}=\{Q_k,\,\widetilde{Q}_k:\,k=1,\,2,\,...\}\) does not satisfy condition \eqref{list:BB3} of a ball-basis. Therefore, \(\mathfrak{Q}\) is not a ball-basis for \((Q, \mathscr{L}^n)\).
			
		\end{list}
		
	\end{remark}

	
	We now examine \eqref{list:BB4}, which enjoys the following property.

	\begin{proposition}\label{P-ball-basis}
		If \(A, B \in \mathfrak{B}\) satisfy \(A \cap B \neq \varnothing\) and \(\mu(A) \leq 2\mu(B)\), then \(A \subset B^\P\).
	\end{proposition}
	
	We next introduce a hierarchical relation among ball-hulls. For a given \(B \in \mathfrak{B}\) and any integer \(n \geq 1\), we set
	\begin{align*}
		B^{[0]} := B, \quad B^{[1]} := B^\P, \quad B^{[n+1]} := (B^{[n]})^\P.
	\end{align*}
	This definition yields \(\mu(B^{[n+1]}) \leq \mathscr{C}_0\,\mu(B^{[n]})\) and consequently \(\mu(B^{[n]}) \leq \mathscr{C}_0^n\,\mu(B)\) for every \(n \geq 0\).

	Before defining a new family of bounded oscillation operators, we first set up the requisite notation. Suppose \((X, \mathfrak{M}, \mu)\) is a measure space endowed with a ball-basis \(\mathfrak{B}\), and let \(\Phi \in \mathscr{Y}\) and \(\phi \in \mathscr{G}\) be as in \eqref{Y function} and \eqref{G function}, respectively. For a ball \(B \in \mathfrak{B}\), we write
	\begin{align*}
		\langle \|f\|\rangle_{\Phi,\phi,B}=\sup_{B'\in\mathfrak{B}:B' \supset B}\|f\|_{\Phi,\phi,B'}, 
	\end{align*}
	where \(\|\cdot\|_{\Phi, \phi, B}\) denotes the generalized Luxemburg norm introduced in \eqref{G-L-norm}.

	\begin{definition}\cite{CFA}
		Let \((X, \mathfrak{M}, \mu)\) be a measure space. We denote by \(\mathscr{L}_0(X, \mathfrak{M}, \mu)\) the collection of all real-valued measurable functions on \(X\), and let \(\mathscr{U}(X, \mathfrak{M}, \mu)\) be a suitable linear subspace thereof. Consider an operator \(T : \mathscr{U}(X, \mathfrak{M}, \mu) \to \mathscr{L}_0(X, \mathfrak{M}, \mu)\).
		
		\begin{itemize}
			\item An operator \(T_l\) is called {\tt linear} if for every \(f, g \in \mathscr{U}(X, \mathfrak{M}, \mu)\) and any scalar \(\alpha \in \mathbb{R}\),
			\begin{align*}
				T_l(\alpha \; f)(x) = \alpha \; T_lf(x),\quad \text{and}\quad T_l(f + g) = T_l(f) + T_l(g).
			\end{align*}
			
			\item An operator \(T_{sl}\) is called {\tt sublinear} if for every \(f, g \in \mathscr{U}(X, \mathfrak{M}, \mu)\) and any \(\alpha \in \mathbb{R}\),
			\begin{align*}
				|T_{sl}(\alpha \; f)(x)| = |\alpha| \; |T_{sl}f(x)|,\quad \text{and}\quad |T_{sl}(f + g)(x)| \leq |T_{sl}f(x)| + |T_{sl}g(x)|.
			\end{align*}
			
			\item A sublinear operator \(T_{sl}\) is said to be {\tt linearizable} if there exists a Banach space \(\mathbb{B}\) together with a \(\mathbb{B}\)-valued linear operator \(\mathscr{T}_{\mathbb{B}} : \mathscr{U}(X, \mathfrak{M}, \mu) \to \mathscr{L}_0(X, \mathfrak{M}, \mu; \mathbb{B})\) such that for each \(f \in \mathscr{U}(X, \mathfrak{M}, \mu)\),
			\begin{align*}
				T_{sl}f(x) = \left\|\mathscr{T}_\mathbb{B}f(x)\right\|_{\mathbb{B}} \quad \text{for } \mu\text{-almost every } x \in X.
			\end{align*}
			
		\end{itemize}
	\end{definition}

	
	We now present the \(\Psi\)-BOOs that will be investigated in the present work.
	
	\begin{definition}\label{def-BOO}
		Let \((X,\,\mathfrak{M},\,\mu)\) be a measure space endowed with a ball-basis \(\mathfrak{B}\), and let \((\Phi,\Psi,\phi,\psi) \in \mathscr{Y} \otimes \mathscr{G}\). Given a Banach space \(\mathbb{B}\), we say that an operator \(T\) is a {\tt \(\mathbb{B}\)-valued \(\Psi\)-bounded oscillation operator} {\rm (\(\Psi\)-BOO)} with respect to \(\mathfrak{B}\) (and \((\Phi,\Psi,\phi,\psi)\in\mathscr{Y} \otimes \mathscr{G}\)) if there exist constants \(\mathscr{C}_1(T), \mathscr{C}_2(T) \in (0, \infty)\) and a \(\mathbb{B}\)-valued linear operator \(\mathscr{T}\) satisfying \[Tf(x) = \|\mathscr{T}f(x)\|_\mathbb{B} \quad (\forall x \in X),\] or, in the case \(\mathbb{B} = \mathbb{R}\), a single real-valued (sub)linear operator \(\mathscr{T} := \{T\}\), such that for every \(f \in L^{(\Psi,\psi)}(X)\subset L^{(\Phi,\phi)}(X)\) the following two conditions hold:
		\begin{list}{\rm (\theenumi)}{\usecounter{enumi}\leftmargin=1.2cm \labelwidth=1cm \itemsep=0.2cm \topsep=.2cm \renewcommand{\theenumi}{\(\Psi\)-BOO-\Roman{enumi}}}
			\item\label{BOO-1} For every \(B_0 \in \mathfrak{B}\) with \(B_0^\P \subsetneq X\), there exists a ball \(B \in \mathfrak{B}\) with \(B \supsetneq B_0\) such that
			\begin{align*}
				\sup_{x \in B_0} \left\| \mathscr{T}(f\,\mathbf{1}_{B^\P})(x) - \mathscr{T}(f\,\mathbf{1}_{B_0^\P})(x) \right\|_{\mathbb{B}}
				\leq \mathscr{C}_1(T) \,\|f\|_{\Phi,\phi,B^\P}.
			\end{align*}
			
			\item\label{BOO-2} For every ball \(B \in \mathfrak{B}\),
			\begin{align*}
				\sup_{x,x' \in B} \left\| \left(\mathscr{T}f - \mathscr{T}(f \,\mathbf{1}_{B^\P})\right)(x) - \left(\mathscr{T}f - \mathscr{T}(f\, \mathbf{1}_{B^\P})\right)(x') \right\|_{ \mathbb{B}} \leq \mathscr{C}_2(T) \, \langle \|f\|\rangle_{\Phi,\phi,B}.
			\end{align*}
		\end{list}
		When \(T\) is real-valued, which is equivalent to \(\mathbb{B} = \mathbb{R}\), we will omit the reference to \(\mathbb{B}\) and the norm \(\|\cdot\|_{\mathbb{B}}\).
	\end{definition}

	\begin{definition}\label{sparse}
		A collection \(\mathcal{S} \subset \mathfrak{B} \) is called {\tt \(\beta\)-sparse}, with \(\beta \in (0, 1)\), if for each \( B \in \mathcal{S} \) there exists a subset \( E_B \subset B \) satisfying the following two properties:
		\begin{itemize}
			\item \(\mu(E_B) \geq \beta \mu(B)\);
			\item the family \(\{E_B : B \in \mathcal{S}\}\) is pairwise disjoint.
		\end{itemize}
	\end{definition}

	
	We now proceed to introduce the Orlicz-Morrey maximal operator, the sparse operator, along with their commutators.

	\begin{definition}\label{OM-S-operator}
		Let \((X, \mathfrak{M}, \mu)\) be a measure space endowed with a ball-basis \(\mathfrak{B}\).
		
		\begin{list}{\rm (\theenumi)}{\usecounter{enumi}\leftmargin=1.2cm \labelwidth=1cm \itemsep=0.2cm \topsep=.2cm \renewcommand{\theenumi}}
			\item[{\rm (OM)}]\label{OM operator} For any \(\Phi \in \mathscr{Y}\) and \(\phi \in \mathscr{G}\), the {\tt Orlicz-Morrey supremum operator} is defined as
			\begin{align}\label{O-M maximal operator}
				\mathcal{M}_{\mathfrak{B}, \Phi, \phi} f(x) := \sup_{B \in \mathfrak{B}: \, x \in B}\|f\|_{\Phi, \phi, B},
			\end{align}
			for \(x \in \bigcup_{B\in\mathfrak{B}}B \subset X\) and we set \(\mathcal{M}_{\mathfrak{B},\Phi,\phi}f(x)=0\) for \(x \notin \bigcup_{B\in\mathfrak{B}}B\).
			
			\item[{\rm (S)}\label{S operator}] Let \(\mathcal{S}\) be a \(\beta\)-sparse family. For \(\Phi \in \mathscr{Y}\), \(\phi \in \mathscr{G}\), and a measurable function \(b\), we then introduce the sparse operator and its commutator counterpart as follows:
			
			\begin{align}
				\label{sparse operator} \mathcal{A}_{\mathcal{S},\Phi,\phi}f(x) & := \sum_{B \in \mathcal{S}} \|f\|_{\Phi,\phi,B} \cdot \mathbf{1}_B(x), \\
				\label{sparse operator commutator} \mathcal{A}^b_{\mathcal{S}, \Phi, \phi}f(x) & := \sum_{\zeta_1 \uplus \zeta_2 = \{1\}} \mathcal{A}^{b, \zeta_1, \zeta_2}_{\mathcal{S}, \Phi, \phi}f(x) = \mathcal{A}^{b, \{1\}, \varnothing}_{\mathcal{S}, \Phi, \phi} f(x) + \mathcal{A}^{b, \varnothing, \{1\}}_{\mathcal{S}, \Phi, \phi} f(x),
			\end{align}
			where
			\begin{align*}
				\mathcal{A}^{b, \{1\}, \varnothing}_{\mathcal{S}, \Phi, \phi} f(x) & := \sum_{B \in \mathcal{S}} |b(x) - b_B| \, \|f\|_{\Phi, \phi, B} \cdot \mathbf{1}_B(x), \\
				\mathcal{A}^{b, \varnothing, \{1\}}_{\mathcal{S}, \Phi, \phi} f(x) & :=\sum_{B \in \mathcal{S}} \|(b(x) - b_B) f \|_{\Phi, \phi, B} \cdot \mathbf{1}_B(x),
			\end{align*}
			and
			\begin{align*}
				b_B := \frac{1}{\phi(B)} \int_B b(y) \phi(y) \, d\mu(y), 
			\end{align*}
		\end{list}
	\end{definition}


	The relationships between the Orlicz-Morrey maximal operator \eqref{O-M maximal operator} and the classical Hardy-Littlewood maximal operator, as well as between the sparse operator \eqref{sparse operator} and classical operators, have already been addressed in our previous work \cite{SZ2026}; we refer the reader to that paper for details. To conclude this subsection, we present the indispensable hypotheses that are required for our main results.

	\begin{definition}\label{Assume condition}
		Let \((X, \mathfrak{M}, \mu)\) be a measure space endowed with a ball-basis \(\mathfrak{B}\).
		
		\begin{list}{\rm (\theenumi)}{\usecounter{enumi}\leftmargin=1.2cm \labelwidth=1cm \itemsep=0.2cm \topsep=.2cm \renewcommand{\theenumi}}
			
			\item[{\rm (\(\mathbb{W}\))}]\label{W condition} Suppose \((\Phi, \Psi, \phi, \psi) \in \mathscr{Y} \otimes \mathscr{G}\). Let \(\mathbb{B}\) be a Banach space and let \(T\) be a \(\mathbb{B}\)-valued sublinear operator which is bounded from \(L^{(\Psi,\psi)}(X)\) into the weak space \(wL^{(\Psi,\psi)}(X)\). We write \(T \in \mathbb{W}_{\Psi, \psi, \lambda}\) (or simply \(T \in \mathbb{W}_{\Psi, \psi}\)) if for every pair of balls \(A, B \in \mathfrak{B}\) with \(A \subset B\), for every function \(f \in L^{(\Psi,\psi)}(X)\), and for any \(0 < \lambda < 1\), the following estimate holds:
			\begin{align}\label{D1.6}
				\mu \left( \left\{x \in A : \|\mathscr{T}(f \, \mathbb{I}_B) (x)\|_\mathbb{B} > \Psi^{-1}\left( \frac{\psi(\mu(A))}{\lambda} \right) \|T\|_{L^{(\Psi, \psi)} \to wL^{(\Psi, \psi)}} \|f\|_{\Phi, \phi, A} \right\} \right) \leq \lambda \mu(A).
			\end{align}
			
			\item[{\rm (\(\mathfrak{N}\))}]\label{N condition} Let \((X, \mathfrak{M}, \mu)\) be a measure space endowed with a ball-basis \(\mathfrak{B}\). The ball-basis \(\mathfrak{B}\) is said to satisfy the {\tt Besicovitch \(\mathfrak{N}\)-condition} with a constant \(\mathfrak{N} \in \mathbb{N}_+\) if for any family \(\mathfrak{A} \subset \mathfrak{B}\), one can extract a subfamily \(\mathfrak{A}' \subset \mathfrak{A}\) such that
			\begin{align}\label{N}
				\bigcup_{B \in \mathfrak{A}} B = \bigcup_{B \in \mathfrak{A}'} B, \quad \text{and} \quad \sum_{B \in \mathfrak{A}'} \mathbb{I}_B(x) \leq \mathfrak{N}, \quad x \in X.
			\end{align}
			
		\end{list}
	\end{definition}

	\subsection{Structure of the article}\label{Section1.4}~
	

	The remainder of this paper is organized as follows. In Sect. \ref{Section2}, we present several typical examples of \(\Psi\)-BOOs within Orlicz-Morrey spaces, including the Orlicz-Morrey maximal operator, \(\theta\)-Calder\'on-Zygmund operators, intrinsic square operators, and Carleson-type operators. Through these examples, we verify the applicability of our main theorems and refine the conditions required in classical results. In Sect. \ref{Section3}, we recall fundamental notions concerning measure spaces endowed with a ball-basis, define the generalized Luxemburg norm, abstract weighted Orlicz-Morrey spaces, and weighted Campanato spaces, and provide several preparatory lemmas. In Sect. \ref{Section4}, we systematically investigate the properties of \(\Psi\)-BOOs, establishing sparse control theorems (Theorem \ref{main}) and weighted norm inequalities (Theorem \ref{main'}) for such operators and their commutators. In Sect. \ref{Section5}, we discuss the duality theory of abstract Orlicz-Morrey spaces endowed with a ball-basis, showing that the dual space can be identified with a block space (Theorem \ref{main2}). The final appendix is divided into three parts: Appendices \ref{Appendix1}-\ref{Appendix3} compile the basic properties of Muckenhoupt weights, Young function class \(\mathscr{Y}\), and \(\mathscr{G}\)-class functions; Appendix \ref{Appendix4} supplies the proofs of several lemmas omitted from the main text; Appendix \ref{Appendix5} provides a detailed overview of five classical types of Orlicz-Morrey spaces (KK-type, Nakai-type, SST-type, DGS-type, and Ho-type), together with their historical background and existing results, and briefly discusses other generalizations.

	\section{Some common examples in Orlicz-Morrey spaces}\label{Section2}


	In this section, we collect several common operators in Orlicz-Morrey spaces. Each of these operators has been shown in \cite{SZ2026} to belong to the class of \(\Phi\)-BOOs.

	\subsection{Orlicz-Morrey maximal operators on measure spaces}\label{Section2.1}~
	

	The Orlicz-Morrey maximal operator is presented in \eqref{O-M maximal operator}.

	\begin{lemma}
		Let \((\Phi, \Psi, \phi, \psi) \in \mathscr{Y} \otimes \mathscr{G}\) and let \((X, \, \mathfrak{M}, \, \mu)\) be a measure space endowed with a ball-basis \(\mathfrak{B}\). Then the operator \(\mathcal{M}_{\mathfrak{B}, \Phi, \phi}\) is a \(\Psi\)-BOO with respect to \(\mathfrak{B}\).
	\end{lemma}
	

	Fix \(B \in \mathfrak{B}\) and pick a non-negative constant \(K\). Observe that
	\begin{align*}
		\|K\|_{\Phi, \phi, B} = \inf \left\{\lambda > 0: \frac{1}{\mu(B) \phi( \mu(B))} \int_B \Phi \left(\frac{K}{\lambda}\right)dx \leq 1\right\} = \frac{K}{\Phi^{-1}(\phi( \mu(B)))}.
	\end{align*}
	This leads to the following hypothesis, which serves as another crucial assumption for the sparse control of \(\Phi\)-BOOs discussed later.
	
	\begin{definition}\label{Emb condition}
		We say that the pair \((\Phi, \phi)\) satisfies the {\tt \(E_{mb}\)-condition}, denoted \((\Phi, \phi) \in E_{mb}\), if there exist constants \(K_1, K_2 > 0\) and some \(\lambda > 0\) such that
		\begin{align*}
			K_1 \leq \Phi^{-1}(\lambda \phi( \mu(B))) \leq K_2,
		\end{align*}
	\end{definition}
	

	Naturally, if \((\Phi, \Psi, \phi, \psi) \in \mathscr{Y} \otimes \mathscr{G}\) and \((\Phi,\phi) \in E_{mb}\), then the maximal operator \(\mathcal{M}_{\mathfrak{B}, \Phi, \phi}\) is bounded from \(L^{(\Psi,\psi)}(\omega)\) to \(L^{(\Psi,\psi)}(\omega)\).

	\subsection{\(\theta\)-Calder\'on-Zygmund operators on \(\mathbb{R}^n\)}\label{Section2.2}~~
	

	Let \(\theta : [0, +\infty) \to [0, +\infty)\) be a modulus of continuity; this means that \(\theta\) is increasing, subadditive, and \(\theta(0)=0\). If \(\theta\) satisfies the {\tt Dini condition}, namely, its {\tt Dini norm}
	\begin{align*}
		\|\theta\|_{\rm Dini} := \int_0^1 \theta(t) \, \frac{dt}{t} < \infty,
	\end{align*}
	then we write \(\theta \in \mathrm{Dini}\). Throughout this subsection, \(\theta\) is always taken to be a modulus of continuity.

	\begin{definition}
		Let \(\omega\) be a modulus of continuity. A kernel \(K(x,y)\) defined on \(\mathbb{R}^n \times \mathbb{R}^n \setminus \{(x,x) : x\in\mathbb{R}^n\}\) is called a {\tt \(\theta\)-Calder\'on-Zygmund kernel} if there exists a positive constant \(C_K\) such that the following three estimates hold:
		
		\begin{list}{\rm (\theenumi)}{\usecounter{enumi}\leftmargin=1.2cm \labelwidth=1cm \itemsep=0.2cm \topsep=.2cm \renewcommand{\theenumi}{\arabic{enumi}}}
			\item {\it Size condition}:
			\begin{align*}
				\left|K(x, y)\right| \leq \frac{C_K}{\left|x-y\right|^n} \quad \text{for} \quad  x \ne y.
			\end{align*}
			
			\item {\it Smoothness condition in the first variable}:
			\begin{align*}
				\left|K(x, y) - K(x', y)\right| \leq \frac{C_K }{\left|x-y\right|^n}\,\theta \left(\frac{|x-x'|}{|x-y|} \right) \quad \text{for} \quad  |x-x'| \leq \frac{1}{2}|x-y|.
			\end{align*}
			
			\item {\it Smoothness condition in the second argument}:
			\begin{align*}
				\left|K(x, y) - K(x, y')\right| \leq \frac{C_K }{\left|x-y\right|^n}\,\theta \left(\frac{|y-y'|}{|x-y|} \right) \quad \text{for} \quad  |y-y'| \leq \frac{1}{2}|x-y|.
			\end{align*}
		\end{list}
		
		In the particular case \(\theta(t)=t^{\delta}\) with \(0<\delta\le1\), such a kernel is called a {\tt standard Calder\'on-Zygmund kernel}.
		
		A linear operator \(T: \mathcal{S} (\mathbb{R}^n) \to \mathcal{S}' (\mathbb{R}^n)\) is called a {\tt \(\theta\)-Calder\'on-Zygmund operator} if the following two conditions are met:
		
		\begin{list}{\rm (\theenumi)}{\usecounter{enumi}\leftmargin=1.2cm \labelwidth=1cm \itemsep=0.2cm \topsep=.2cm \renewcommand{\theenumi}{\roman{enumi}}}
			
			\item \(T\) is bounded on \(L^2(\mathbb{R}^n)\); 
			
			\item There exists a \(\theta\)-Calder\'on-Zygmund kernel \(K\) such that for all \(f\in L^2_{\text{\rm comp}}(\mathbb{R}^n)\),
			\begin{align*}
				Tf(x)=\int_{\mathbb{R}^n} K(x,y)f(y)\,dy,\qquad x \notin {\rm supp}\,f.
			\end{align*}
			
			\end{list}
		\end{definition}

	\begin{lemma}
		Assume that \((\Phi, \Psi, \phi, \psi) \in \mathscr{Y} \otimes \mathscr{G}\). If \(T\) is a \(\theta\)-Calder\'on-Zygmund operator with \(\omega \in {\rm Dini}\), then \(T\) is a \(\Psi\)-BOO with respect to \(\mathfrak{B}_E\) (see \eqref{list:R1}), with constants  
		\begin{align*}
			\mathscr{C}_1(T) \lesssim C_K \quad \text{and} \quad \mathscr{C}_2(T) \lesssim C_K\|\omega\|_{\rm Dini}.
		\end{align*}
	\end{lemma}

	From Theorem \ref{main} together with Theorem \ref{main'}, we deduce the following conclusion.

	\begin{theorem}\label{example-CZ}
		Let \(\lambda > 3\mathscr{C}_0^6\). Suppose \(T\) is a \(\theta\)-Calder\'on-Zygmund operator as defined above, and assume that \(T \in \mathbb{W}_{\Psi, \psi, \lambda}\) and \(b \in \mathfrak{L}^1_\psi(\omega)\). Furthermore, let \((\Phi,\Psi,\phi,\psi) \in \mathscr{Y} \otimes \mathscr{G}\) and \((\Phi, \phi) \in E_{mb}\). Then the following conclusions hold.
		
		\begin{list}{\rm (\theenumi)}{\usecounter{enumi}\leftmargin=1.2cm \labelwidth=1cm \itemsep=0.2cm \topsep=.2cm \renewcommand{\theenumi}{\alph{enumi}}}
			\item There exist two families \(\mathcal{S}_1, \mathcal{S}_2 \subset \mathfrak{B}_E\), each of which is \(\frac{1}{2\mathscr{C}_0^3}\)-sparse, such that for every function \(f \in L^{(\Psi,\psi)}(\mathbb{R}^n) \subset L^{(\Phi,\phi)}(\mathbb{R}^n)\), for every ball \(B \in \mathfrak{B}_E\), and \(a.e. \, x \in B\),
			\begin{align*}
				|T f(x)| \lesssim \left(\mathscr{C}_1(T) + \mathscr{C}_2(T) + \|T\|_{L^{(\Psi, \psi)}(\mathbb{R}^n) \to wL^{(\Psi,\psi)}(\mathbb{R}^n)} \right) \cdot \left[ \mathcal{A}_{\mathcal{S}_1, \Phi,\phi} f(x) + \mathcal{A}_{\mathcal{S}_2, \Phi,\phi} f(x) \right],
			\end{align*}
			and
			\begin{align*}
				|[T, b]f(x)| \lesssim \left(\mathscr{C}_1(T) + \mathscr{C}_2(T) + \|T\|_{L^{(\Psi, \psi)}(\mathbb{R}^n) \to wL^{(\Psi,\psi)}(\mathbb{R}^n)} \right) \cdot \left[ \mathcal{A}^b_{\mathcal{S}_1, \Phi,\phi} f(x) + \mathcal{A}^b_{\mathcal{S}_2, \Phi,\phi} f(x) \right].
			\end{align*}
			
			\item Additionally, if \(\mathfrak{B}_E\) fulfills the Besicovitch \(\mathfrak{N}\)-condition, then
			\begin{align*}
				\|T\|_{L^{(\Psi, \psi)}(\omega) \to L^{(\Psi,\psi)}(\omega)} \lesssim \mathscr{K} \, \mathfrak{N} \, \left(\mathscr{C}_1(T) + \mathscr{C}_2(T) + \|T\|_{L^{(\Psi, \psi)}(\mathbb{R}^n) \to wL^{(\Psi,\psi)}(\mathbb{R}^n)} \right),
			\end{align*}
			and
			\begin{align*}
				\|[T, b]\|_{L^{(\Psi, \psi)}(\omega) \to L^{(\Psi,\psi)}(\omega)} \lesssim \mathscr{K} \, \mathfrak{N} \, \left(\mathscr{C}_1(T) + \mathscr{C}_2(T) + \|T\|_{L^{(\Psi, \psi)}(\mathbb{R}^n) \to wL^{(\Psi,\psi)}(\mathbb{R}^n)} \right) \|b\|_{\mathfrak{L}^1_\psi(\omega)}.
			\end{align*}
		\end{list}
	\end{theorem}

	\subsection{Intrinsic square operators on \(\mathbb{R}^n\)}\label{Section2.3}~~

	In the present subsection, we keep a fixed exponent \(0 < \alpha \leq 1\). By \({\rm Lip}_{\alpha}\) we mean the collection of all Lipschitz functions \(\varphi\) on \(\mathbb{R}^n\) with homogeneous norm \(1\), supported in the unit ball \(\{ x \in \mathbb{R}^n : |x| \leq 1 \}\), and satisfying the zero integral condition \(\int_{\mathbb{R}^{n}} \varphi(x) \, dx = 0\). For \(f \in L^{1}_{{\rm loc}}(\mathbb{R}^{n})\) and a point \((y, t) \in \mathbb{R}^n \times \mathbb{R}_+ := \mathbb{R}^{n+1}_+\), the quantity \(A_{\alpha}f(t,y)\) is defined as
	\begin{align*}
		A_{\alpha}f(t, y) \;:=\; \sup_{\varphi \in {\rm Lip}_{\alpha}} \left| (\varphi_t * f)(y) \right|,
	\end{align*}
	where \(\varphi_t(x) = t^{-n} \varphi(x / t)\), and we call \(\varphi_t\) the kernel of \(A_{\alpha}\).

	Let \(\Gamma_\beta(x) := \{(y, t) \in \mathbb{R}^{n+1}_+ : |x-y| < \beta t\}\) for \(\beta \in (0, \infty)\). We introduce the following three {\tt intrinsic square operators}:
	\begin{align}
		\label{G 1} g_{\alpha}f(x) & := \left(\int_{0}^{\infty} \left( A_\alpha f(t, x) \right)^2\frac{dt}{t} \right)^{\frac{1}{2}},\\
		\label{G 2} g_{\alpha, \beta}f(x) & := \left( \iint_{\Gamma_\beta(x)} \left( A_\alpha f(t, z) \right)^2 \frac{dzdt}{t^{n+1}} \right)^{\frac{1}{2}},\\
		\label{G 3} g^*_{\lambda,\alpha}f(x) & := \left(\iint_{\mathbb{R}^{n+1}_+}\left(\frac{t}{t+|x-z|}\right)^{n\lambda}\left( A_\alpha f(t, z) \right)^2\frac{dzdt}{t^{n+1}} \right)^{\frac{1}{2}}, \quad \lambda > 3+\frac{2\alpha}{n}.
	\end{align}
	
	To simplify notation, we introduce the following three Banach spaces:
	\begin{align}
		\label{U 1} \mathbb{B}_{g_{\alpha}} & := \left\{ F_1 : \mathbb{R}_+ \to \mathbb{R} \;\big|\; \left\|f_1\right\|_{\mathbb{B}_{g_{\alpha}}} := \left( \int_0^\infty |f_1(t)|^2 \frac{dt}{t} \right)^{\frac{1}{2}} < \infty \right\},\\
		\label{U 2} \mathbb{B}_{g_{\alpha, \beta}} & := \left\{ F_2 : \mathbb{R}^{n+1}_+ \to \mathbb{R} \;\big|\; \left\|f_2\right\|_{\mathbb{B}_{g_{\alpha, \beta}}} := \left( \int_{\mathbb{R}^{n+1}_+} |f_2(z,t)|^2 \frac{dz dt}{t^{n+1}} \right)^{\frac12} < \infty \right\},\\
		\label{U 3} \mathbb{B}_{g^*_{\lambda,\alpha}} & := \mathbb{B}_{g_{\alpha, \beta}}.
	\end{align}
	
	We proceed as follows for each of the three intrinsic square operators mentioned above.
	
	{\bf (1)} For the operator \(g_{\alpha}\): Let
	\begin{align}\label{Iso-Kernel 1}
		\mathscr{P}_{g_{\alpha}}(x) := \{ \varphi_t(x) \}_{t>0},
	\end{align}
	where \(\varphi \in {\rm Lip}_{\alpha}\). Further, denote
	\begin{align}\label{Iso-Tf 1}
		\mathscr{A}_{g_{\alpha}}f(x):=\{A_\alpha f(t, x)\}_{t>0}.
	\end{align}
	Here \(\mathscr{P}_{g_{\alpha}}\) is the kernel of \(\mathscr{A}_{g_{\alpha}}\). We then have \(g_{\alpha}f(x)=\left\|\mathscr{A}_{g_{\alpha}}f(x)\right\|_{\mathbb{B}_{g_{\alpha}}}\).
	
	{\bf (2)} For the operator \(g_{\alpha, \beta}\): Take \(P_{z,t}(x)=\mathbb{I}_{ \Gamma_\beta(0)}(z)\varphi_t(x-z)\), where \(\varphi \in {\rm Lip}_{\alpha}\) and \((z,t)\in\mathbb{R}^{n+1}_+\). Define
	\begin{align}\label{Iso-Kernel 2}
		\mathscr{P}_{g_{\alpha, \beta}}(x):=\{P_{z,t}(x)\}_{(z,t)\in\mathbb{R}^{n+1}_+}
	\end{align}
	and 
	\begin{align}\label{Iso-Tf 2}
		\mathscr{A}_{g_{\alpha, \beta}}f(x):=\{A_\alpha f(P_{z,t}, x)\}_{(z,t) \in \mathbb{R}^{n+1}_+}=\sup \left| (P_{z,t} * f)(x) \right|,
	\end{align}
	with \(\mathscr{P}_{g_{\alpha, \beta}}\) being the kernel of \(\mathscr{A}_{g_{\alpha, \beta}}\). Then \(g_{\alpha, \beta}f(x)=\|\mathscr{A}_{g_{\alpha, \beta}}f(x)\|_{\mathbb{B}_{g_{\alpha, \beta}}}\).
	
	{\bf (3)} For the operator \(g^*_{\lambda,\alpha}\):  
	Following a similar procedure, define
	\[
	Q_{z,t}(x)=\left(\frac{t}{t+|z|}\right)^{\frac{n \lambda}{2}} \varphi_t(x-z),
	\quad \varphi \in {\rm Lip}_{\alpha},\,\, (z,t)\in \mathbb{R}^{n+1}_+.
	\]
	Set
	\begin{align}\label{Iso-Kernel 3}
		\mathscr{P}_{g^*_{\lambda,\alpha}}(x) := \{Q_{z,t}(x)\}_{(z,t) \in \mathbb{R}^{n+1}_+}
	\end{align}
	and 
	\begin{align}\label{Iso-Tf 3}
		\mathscr{A}_{g^*_{\lambda,\alpha}}f(x):=\{A_\alpha f(Q_{z,t}, x)\}_{(z,t) \in \mathbb{R}^{n+1}_+}=\sup \left| (Q_{z,t} * f)(x) \right|.
	\end{align}
	Consequently, \(g^*_{\lambda,\alpha}\) can be expressed as \(g^*_{\lambda,\alpha}f(x)=\|\mathscr{A}_{g^*_{\lambda,\alpha}}f(x)\|_{\mathbb{B}_{g^*_{\lambda,\alpha}}}\).
	
	%

	For the remaining part of this subsection, we consider \(\mathfrak{B}_Q\), the set of all cubes in \(\mathbb{R}^n\) aligned with the coordinate axes. According to \eqref{list:R2}, \(\mathfrak{B}_Q\) is a ball-basis on \((\mathbb{R}^n, \mathscr{L}^n)\). Also, every \(Q\in\mathfrak{B}_Q\) satisfies \(Q^\P = 5Q\).

	We begin by noting H\"older's inequality in Orlicz-Morrey spaces.
	\begin{lemma}\cite[Lemma 9.2]{OM-Nakai2008}
		Given a Young function \(\Phi\), a function \(\phi \in \mathcal{G}\), and a set \(Q \in\mathfrak{B}_Q\), the following integral inequality holds:
		\begin{align*}
			\int_{Q} f(x)g(x) \, dx \leq 2|Q|\phi(|Q|)\left\| f\right\|_{\Phi,\phi,Q}\left\|g\right\|_{\widetilde{\Phi},\phi,Q},
		\end{align*}
		where \(\widetilde{\Phi}\) denotes the complementary Young function of \(\Phi\).
	\end{lemma}
	
	\begin{definition}
		Let \(\Phi \in \mathscr{Y}\) and \(\phi \in \mathscr{G}\), and let \(\widetilde{\Phi}\) be the complementary function of \(\Phi\). We say that a kernel \(\mathscr{P}\) satisfies the \(\mathbb{B}\)-valued \(\Phi\)-\(\phi\)-H\"ormander condition if
		\begin{align*}
			\sup_{Q \in\mathfrak{B}_Q, \, x \in \frac{1}{2} Q}
			|Q|\phi(|Q|) \Bigl\|\left\| \mathscr{P}(x) \right\|_{\mathbb{B}}\Bigr\|_{\widetilde{\Phi},\phi,2Q \setminus Q} < \infty,
		\end{align*}
		and 
		\begin{align*}
			\sup_{\substack{Q\in\mathfrak{B}_Q \\ x,x' \in \frac{1}{2} Q}} \sum _{j=1}^{\infty} |2^jQ|\phi(|2^jQ|)\Bigl\| \left\| \mathscr{P}(x) - \mathscr{P}(x') \right\|_{\mathbb{B}} \Bigr\|_{ \widetilde{\Phi}, \phi, \widetilde{B}_j(Q)} < \infty.
		\end{align*}
		where \(\widetilde{B}_j(Q) := 2^j Q \setminus 2^{j-1} Q\) for \(j= 1,2,\dots\). Here, \(\mathscr{P}(x)\) takes values as given in (\ref{Iso-Kernel 1}), (\ref{Iso-Kernel 2}), and (\ref{Iso-Kernel 3}), and \(\mathbb{B}\) takes values as in (\ref{U 1}), (\ref{U 2}), and (\ref{U 3}).
	\end{definition}

	\begin{lemma}
		Suppose that \((\Phi,\Psi,\phi,\psi)\in\mathscr{Y} \otimes \mathscr{G}\). Let \(\mathscr{P}(x)\) be the intrinsic square operators satisfying the \(\mathbb{B}\)-valued \(\Phi\)-\(\phi\)-H\"ormander condition, where \(\mathfrak{G}(x)\) is taken from (\ref{G 1}), (\ref{G 2}) and (\ref{G 3}), and \(\mathbb{B}\) is taken from (\ref{U 1}), (\ref{U 2}) and (\ref{U 3}). Then \(\mathfrak{G}(x)\) is a \(\mathbb{B}\)-valued \(\Psi\)-BOO with respect to \(\mathfrak{B}_Q\).
	\end{lemma}

	A direct application of Theorem \ref{main} and Theorem \ref{main'} gives the following.
	
	\begin{theorem}\label{example-intrinsic}
		Let \(\mathfrak{B}_E\) be a ball-basis in \(\mathbb{R}^n\) and let \(\lambda > 3\mathscr{C}_0^6\). Suppose that \((\Phi,\Psi,\phi,\psi) \in \mathscr{Y} \otimes \mathscr{G}\) and \((\Phi, \phi) \in E_{mb}\). Denote by \(\mathfrak{G}(x)\) the object defined in (\ref{G 1}), (\ref{G 2}) and (\ref{G 3}), and by \(\mathbb{B}\) the one defined in (\ref{U 1}), (\ref{U 2}) and (\ref{U 3}). We then have the following:
		
		\begin{list}{\rm (\theenumi)}{\usecounter{enumi}\leftmargin=1.2cm \labelwidth=1cm \itemsep=0.2cm \topsep=.2cm \renewcommand{\theenumi}{\alph{enumi}}}
			\item There exist two \(\frac{1}{2 \mathscr{C}_0^3}\)-sparse families \(\mathcal{S}_1, \, \mathcal{S}_2 \subset \mathfrak{B}_E\) such that for every \(f \in L^{(\Psi,\psi)}(\mathbb{R}^n)\subset L^{(\Phi,\phi)}(\mathbb{R}^n)\), every \(B \in \mathfrak{B}_E\), and \(a.e. \, x \in B\),
			\begin{align*}
				\|\mathscr{A} f(x)\|_\mathbb{B} \lesssim \left(\mathscr{C}_1(\mathfrak{G}) + \mathscr{C}_2(\mathfrak{G}) + \|\mathfrak{G}\|_{L^{(\Psi, \psi)}(\mathbb{R}^n) \to wL^{(\Psi,\psi)}(\mathbb{R}^n)} \right) \cdot \left[ \mathcal{A}_{\mathcal{S}_1, \Phi,\phi} f(x) + \mathcal{A}_{\mathcal{S}_2, \Phi,\phi} f(x) \right],
			\end{align*}
			and
			\begin{align*}
				\|[\mathscr{A}, b] f(x)\|_\mathbb{B} \lesssim \left(\mathscr{C}_1(\mathfrak{G}) + \mathscr{C}_2(\mathfrak{G}) + \|\mathfrak{G}\|_{L^{(\Psi, \psi)}(\mathbb{R}^n) \to wL^{(\Psi,\psi)}(\mathbb{R}^n)} \right) \cdot \left[ \mathcal{A}^b_{\mathcal{S}_1, \Phi,\phi} f(x) + \mathcal{A}^b_{\mathcal{S}_2, \Phi,\phi} f(x) \right].
			\end{align*}
			where \(\mathscr{A}f(x)\) is taken from (\ref{Iso-Tf 1}), (\ref{Iso-Tf 2}) and (\ref{Iso-Tf 3}).

			\item If, in addition, \(\mathfrak{B}_E\) satisfies the Besicovitch \(\mathfrak{N}\)-condition, then
			\begin{align*}
				\|\mathfrak{G}\|_{L^{(\Psi, \psi)}(\omega) \to L^{(\Psi,\psi)}(\omega)} \lesssim \mathscr{K} \, \mathfrak{N} \, \left(\mathscr{C}_1(\mathfrak{G}) + \mathscr{C}_2(\mathfrak{G}) + \|\mathfrak{G}\|_{L^{(\Psi, \psi)}(\mathbb{R}^n) \to wL^{(\Psi,\psi)}(\mathbb{R}^n)} \right).
			\end{align*}
			and
			\begin{align*}
				\|[\mathfrak{G}, b]\|_{L^{(\Psi, \psi)}(\omega) \to L^{(\Psi,\psi)}(\omega)} \lesssim \mathscr{K} \, \mathfrak{N} \, \left(\mathscr{C}_1(\mathfrak{G}) + \mathscr{C}_2(\mathfrak{G}) + \|\mathfrak{G}\|_{L^{(\Psi, \psi)}(\mathbb{R}^n) \to wL^{(\Psi,\psi)}(\mathbb{R}^n)} \right) \|b\|_{\mathfrak{L}^1_\psi(\omega)}.
			\end{align*}
		\end{list}
	\end{theorem}

	\subsection{Carleson-type operators on measure spaces}\label{Section2.4}~~

	Let \((X, \, \mathfrak{M}, \, \mu)\) be a measure space endowed with a ball-basis \(\mathfrak{B}\). Consider a family of \(\Psi\)-BOOs \(\{T_{\alpha}\}_{\alpha \in \mathfrak{A}}\) on \((X, \, \mathfrak{M}, \, \mu)\), where \(\mathfrak{A}\) is an index set. We can define a Carleson-type operator of the form
	\[
	T^{\mathfrak{A}}f(x):=\sup_{\alpha \in \mathfrak{A}} \left| T_{\alpha}f(x) \right|.
	\]
	Here, we define a Banach space as follows:  
	\[
	\mathbb{B}:=\left\{ F: \mathfrak{A} \to \mathbb{R} \;\big|\; \left\|f\right\|_{\mathbb{B}} := \sup_{\alpha \in \mathfrak{A}} \left|f(\alpha)\right| < \infty \right\}.
	\]
	Then, according to Definition \ref{def-BOO}, the following theorem holds trivially.

	\begin{theorem}
		Let \((\Phi,\Psi,\phi,\psi)\in\mathscr{Y} \otimes \mathscr{G}\) and let \((X,\,\mathfrak{M},\,\mu)\) be a measure space endowed with a ball-basis \(\mathfrak{B}\). Suppose that \(\{T_{\alpha}\}_{\alpha \in \mathfrak{A}}\) is a family of \(\Psi\)-BOOs satisfying  
		\begin{align}\label{Carleson-type-C}
			\mathscr{C}_1^{\mathfrak{A}} := \sup_{\alpha \in \mathfrak{A}} \mathscr{C}_1 (T_{\alpha}) < \infty \quad \text{and} \quad \mathscr{C}_2^{\mathfrak{A}} := \sup_{\alpha \in \mathfrak{A}}\mathscr{C}_2(T_{\alpha}) < \infty.
		\end{align}
		Then \(T^{\mathfrak{A}}\) is a \(\mathbb{B}\)-valued \(\Psi\)-BOO with respect to \(\mathfrak{B}\), and its constants satisfy  
		\[
		\mathscr{C}_1(T^{\mathfrak{A}}) \le \mathscr{C}_1^{\mathfrak{A}} \quad \text{and} \quad \mathscr{C}_2(T^{\mathfrak{A}}) \le \mathscr{C}_2^{\mathfrak{A}} .
		\]
	\end{theorem}
	
	Combining Theorem \ref{main} and Theorem \ref{main'}, we arrive at the following result.
	\begin{theorem}\label{example-Carleson}
		Let \((X,\,\mathfrak{M},\,\mu)\) be a measure space endowed with a ball-basis \(\mathfrak{B}\) and let \(\lambda > 3\mathscr{C}_0^6\). Suppose that \((\Phi,\Psi,\phi,\psi)\in\mathscr{Y} \otimes \mathscr{G}\) and \((\Phi, \phi) \in E_{mb}\). Let \(\{T_{\alpha}\}_{\alpha \in \mathfrak{A}}\) be a family of \(\Psi\)-BOOs satisfying (\ref{Carleson-type-C}), and \(T^{\mathfrak{A}} \in \mathbb{W}_{\Psi, \psi, \lambda}\) and \(b \in \mathfrak{L}^1_\psi(\omega)\). The following assertions are valid:
		\begin{list}{\rm (\theenumi)}{\usecounter{enumi}\leftmargin=1.2cm \labelwidth=1cm \itemsep=0.2cm \topsep=.2cm \renewcommand{\theenumi}{\alph{enumi}}}
			\item There exist two \(\frac{1}{2 \mathscr{C}_0^3}\)-sparse families \(\mathcal{S}_1, \, \mathcal{S}_2 \subset \mathfrak{B}\) such that for every \(f \in L^{(\Psi,\psi)}(X)\subset L^{(\Phi,\phi)}(X)\), every \(B \in \mathfrak{B}\), and \(a.e. \, x \in B\),
			\begin{align*}
				|T^{\mathfrak{A}} f(x)| \lesssim \left(\mathscr{C}_1(T^{\mathfrak{A}}) + \mathscr{C}_2(T^{\mathfrak{A}}) + \|T^{\mathfrak{A}}\|_{L^{(\Psi, \psi)}(\mathbb{R}^n) \to wL^{(\Psi,\psi)}(\mathbb{R}^n)} \right) \cdot \left[ \mathcal{A}_{\mathcal{S}_1, \Phi,\phi} f(x) + \mathcal{A}_{\mathcal{S}_2, \Phi,\phi} f(x) \right],
			\end{align*}
			and
			\begin{align*}
				|[T^{\mathfrak{A}}, b]f(x)| \lesssim \left(\mathscr{C}_1(T^{\mathfrak{A}}) + \mathscr{C}_2(T^{\mathfrak{A}}) + \|T^{\mathfrak{A}}\|_{L^{(\Psi, \psi)}(\mathbb{R}^n) \to wL^{(\Psi,\psi)}(\mathbb{R}^n)} \right) \cdot \left[ \mathcal{A}^b_{\mathcal{S}_1, \Phi,\phi} f(x) + \mathcal{A}^b_{\mathcal{S}_2, \Phi,\phi} f(x) \right].
			\end{align*}

			\item If, furthermore, \(\mathfrak{B}\) satisfies the Besicovitch \(\mathfrak{N}\)-condition, then
			\begin{align*}
				\|T^{\mathfrak{A}}\|_{L^{(\Psi, \psi)}(\omega) \to L^{(\Psi,\psi)}(\omega)} \lesssim \mathscr{K} \, \mathfrak{N} \, \left(\mathscr{C}_1(T^{\mathfrak{A}}) + \mathscr{C}_2(T^{\mathfrak{A}}) + \|T^{\mathfrak{A}}\|_{L^{(\Psi, \psi)}(\mathbb{R}^n) \to wL^{(\Psi,\psi)}(\mathbb{R}^n)} \right),
			\end{align*}
			and
			\begin{align*}
				\|[T^{\mathfrak{A}}, b]\|_{L^{(\Psi, \psi)}(\omega) \to L^{(\Psi,\psi)}(\omega)} \lesssim \mathscr{K} \, \mathfrak{N} \, \left(\mathscr{C}_1(T^{\mathfrak{A}}) + \mathscr{C}_2(T^{\mathfrak{A}}) + \|T^{\mathfrak{A}}\|_{L^{(\Psi, \psi)}(\mathbb{R}^n) \to wL^{(\Psi,\psi)}(\mathbb{R}^n)} \right) \|b\|_{\mathfrak{L}^1_\psi(\omega)}.
			\end{align*}
		\end{list}
	\end{theorem}

	\section{Preliminaries}\label{Section3}
	
	\subsection{Generalized mean oscillation}\label{Section3.1}~
	
	For a measurable subset \(\Omega \subset X\) and a measurable function \(f\), together with a weight \(\omega\) and a positive number \(t\), we set
	\begin{align}\label{2.1}
		\omega(\Omega,f,t) = \omega(\{x \in \Omega:|f(x)|>t\}).
	\end{align}
	When \(\Omega = X\), we abbreviate \(\omega(f, t)\) for \(\omega(\Omega, f, t)\).

	Next, we introduce two variants of the Luxemburg norm. Let \((X, \mathfrak{M}, \mu)\) be a measure space endowed with a ball-basis \(\mathfrak{B}\). For a Young function \(\Phi \in \mathscr{Y}\), a function \(\phi \in \mathscr{G}\), and a ball \(B \in \mathfrak{B}\), we define
	\begin{align}
		\label{2.2} \|f\|_{\Phi, \phi, \omega, B} = & \inf \left\{ \lambda > 0 : \frac{1}{\omega(B) \phi (\omega(B))} \int_{B} \Phi \left( \frac{|f(x)|}{\lambda}\right) \omega(x) d\mu(x) \leq 1 \right\}, \\
		\label{2.3} \|f\|_{\Phi, \phi, \omega, B, \text{weak}} = & \inf \left\{ \lambda > 0 : \sup_{t>0} \frac{t \omega(B, \Phi(|f|/\lambda), t)}{\omega(B) \phi(\omega(B))} \leq 1\right\}.
	\end{align}

	\begin{remark}
		From \eqref{2.2} and \eqref{2.3} we readily obtain the following facts:
		\begin{itemize}
			\item \eqref{2.2} implies the estimate
			\begin{align}\label{remark-2.1}
				\frac{1}{\omega(B) \phi (\omega(B))} \int_{B} \Phi \left( \frac{|f(x)|}{\|f\|_{\Phi, \phi, \omega, B}}\right) \omega(x) d\mu(x) \leq 1,
			\end{align}
			
			\item Combining \eqref{2.2} and \eqref{2.3} yields
			\[
			\|f\|_{\Phi, \phi, \omega, B, \text{weak}} \leq \|f\|_{\Phi, \phi, \omega, B}.
			\]
			
			\item For \eqref{2.3}, invoking \eqref{2.1} gives the identity
			\[
			\sup_{t>0} t \, \omega(\Omega, \Phi(|f|), t) = \sup_{t>0} t \, \omega(\Omega, f, \Phi^{-1}(t)) = \sup_{t>0} \Phi(t) \, \omega(\Omega, f, t).
			\]
			
		\end{itemize}
	\end{remark}
	
	
	Setting \(\omega(x) \equiv 1\) in \eqref{2.2} recovers the generalized Luxemburg norm introduced in \cite{SZ2026}, which is precisely the generalized mean oscillation used throughout this paper.
	
	\begin{align}\label{G-L-norm}
		\|f\|_{\Phi, \phi, 1, B} = \inf \left\{ \lambda > 0 : \frac{1}{\mu(B) \phi (\mu(B))} \int_{B} \Phi \left( \frac{|f(x)|}{\lambda}\right) d\mu(x) \leq 1 \right\} := \|f\|_{\Phi, \phi, B}.
	\end{align}
	
	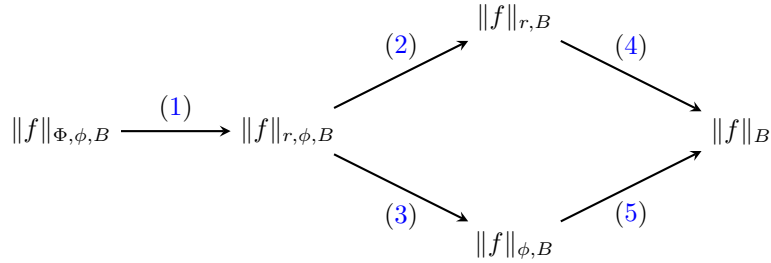
\begin{figure}[!h]
		\begin{center}
			\begin{tikzpicture}
				\node (1) at (-3, 0) {\(\|f\|_{\Phi, \phi, B}\)};
				\node (2) at (0, 0) {\(\|f\|_{r, \phi, B}\)};
				\node (3) at (3, 1.5) {\(\|f\|_{r, B}\)};
				\node (4) at (3, -1.5) {\(\|f\|_{\phi, B}\)};
				\node (5) at (6, 0) {\(\|f\|_{B}\)};
				\draw[-stealth, thick] (1)--(2) node[midway, above] {\eqref{L_1}};
				\draw[-stealth, thick] (2)--(3) node[midway, above = 2pt] {\eqref{L_2}};
				\draw[-stealth, thick] (2)--(4) node[midway, below = 2pt] {\eqref{L_3}};
				\draw[-stealth, thick] (3)--(5) node[midway, above = 2pt] {\eqref{L_4}};
				\draw[-stealth, thick] (4)--(5) node[midway, below = 2pt] {\eqref{L_5}};
			\end{tikzpicture}
		\end{center}
		\caption{Relations among different mean oscillations}\label{figure3}
	\end{figure}
	
	\begin{remark}
		The detailed content of Figure \ref{figure3} is as follows:
		\begin{list}{\rm (\theenumi)}{\usecounter{enumi}\leftmargin=1.2cm \labelwidth=1cm \itemsep=0.2cm \topsep=.2cm \renewcommand{\theenumi}{\arabic{enumi}}}
			
			\item \label{L_1} If we take \(\Phi(t) = t^r\), then we get
			\[ \|f\|_{\Phi(t) = t^r, \phi, B} = \left(\frac{1}{\mu(B) \phi(\mu(B))} \int_{B} |f|^r d \mu \right)^{\frac1r} =: \|f\|_{r, \phi, B},\]
			
			\item \label{L_2} In the case \(\phi(t) = 1\), it follows that
			\[ \|f\|_{r, \phi(t) = 1, B} = \left( \frac{1}{\mu(B)}\int_{B} |f|^r d \mu \right)^{\frac1r} =: \|f\|_{r, B},\]
			
			\item  \label{L_3} For the choice \(r = 1\), we have
			\[ \|f\|_{r = 1, \phi, B} = \frac{1}{\mu(B) \phi(\mu(B))} \int_{B} |f| d \mu =: \|f\|_{\phi, B},\]
			
			\item  \label{L_4} Setting \(r = 1\) yields
			\[ \|f\|_{r = 1, B} = \frac{1}{\mu(B)} \int_{B} |f| d \mu =: \|f\|_B,\]
			
			\item  \label{L_5} When \(\phi(t) = 1\), one arrives at
			\[ \|f\|_{\phi(t) = 1, B} = \frac{1}{\mu(B)} \int_{B} |f| d \mu =: \|f\|_B.\]
		\end{list}
	\end{remark}
	
	
	From \eqref{2.2} we can derive the following lemma; its proof is provided in Appendix \ref{Appendix4}.
	
	\begin{lemma}\label{lemma:AB}
		Let \((X, \mathfrak{M}, \mu)\) be a measure space endowed with a ball-basis \(\mathfrak{B}\), and let \(\omega \in A^\mathfrak{B}_p\) (see Appendix \ref{Appendix1}). Take a Young function \(\Phi \in \mathscr{Y}\) and a function \(\phi \in \mathscr{G}\). Suppose \(A, B \in \mathfrak{B}\). Then the following estimates hold.
		
		\begin{list}{\rm (\theenumi)}{\usecounter{enumi}\leftmargin=1.2cm \labelwidth=1cm \itemsep=0.2cm \topsep=.2cm \renewcommand{\theenumi}{\alph{enumi}}}
			\item \label{AB-1} If \(A \cap B \neq \varnothing\) and \(\omega(B) \le C_{A,B}\,\omega(A)\) for some constant \(C_{A,B} \ge 1\), then
			\[
			\|f \cdot \mathbf{1}_B\|_{\Phi, \phi, \omega, A} \le \max\{1, C_{dc}^{-1}\}\,\max\{1, C_{ic}\}\,C_{A,B} \; \|f\|_{\Phi, \phi, \omega, B}.
			\]
			
			\item \label{AB-2} If \(A \subset B\), we have
			\[
			\|f \cdot \mathbf{1}_A\|_{\Phi, \phi, \omega, B} \le \max\{1, C_{ic}\} \; \|f\|_{\Phi, \phi, \omega, A},
			\]
			and also
			\[
			\|f\|_{\Phi, \phi, \omega, A} \le \frac{\max\{1, C_{ic}\}\; \omega(B)\,\phi(\omega(B))}{\omega(A)\,\phi(\omega(A))} \; \|f\|_{\Phi, \phi, \omega, B}.
			\]
		\end{list}
	\end{lemma}

	\subsection{Abstract weighted (weak) Orlicz-Morrey space endowed with a ball-basis}\label{Section3.2}~
	

	We now introduce the abstract weighted Orlicz-Morrey space equipped with a ball-basis.
	
	\begin{definition}\label{weight Orlicz-Morrey space}
		Let \((X, \mathfrak{M}, \mu)\) be a measure space endowed with a ball-basis \(\mathfrak{B}\), let \(\Phi \in \mathscr{Y}\) be a Young function and \(\phi \in \mathscr{G}\), and suppose \(\omega \in A^\mathfrak{B}_p\). Then the {\tt abstract weighted Orlicz-Morrey space endowed with a ball-basis} is defined as
		\begin{align*}
			L^{(\Phi,\phi)} (X, \omega) = \left\{ f \in L_{{\rm loc}}^1(X, \omega(x) dx) : \|f\|_{L^{(\Phi,\phi)}(X, \omega)} < +\infty \right\} =: L^{(\Phi,\phi)} (\omega),
		\end{align*}
		where the norm is given by
		\begin{align*}
			\|f\|_{L^{(\Phi, \phi)}(X, \omega)} = \sup_{B \in \mathfrak{B}}\|f\|_{\Phi, \phi, \omega, B} =: \|f\|_{L^{(\Phi,\phi)} (\omega)}.
		\end{align*}
	\end{definition}

	
	We now define the abstract weighted weak Orlicz-Morrey space endowed with a ball-basis.
	
	\begin{definition}\label{weak Orlicz-Morrey space}
		Let \((X, \mathfrak{M}, \mu)\) be a measure space endowed with a ball-basis \(\mathfrak{B}\), let \(\Phi \in \mathscr{Y}\) and \(\phi \in \mathscr{G}\), and suppose \(\omega \in A^\mathfrak{B}_p\). Then the {\tt abstract weighted weak Orlicz-Morrey space endowed with a ball-basis} is defined by
		\begin{align*}
			L^{(\Phi,\phi)}_{\rm weak} (X, \omega) = \left\{ f \in L_{\rm loc}^1 (X, \omega(x) dx) : \|f\|_{L^{(\Phi,\phi)}_{\rm weak} (X, \omega)} < +\infty \right\} =: L^{(\Phi,\phi)}_{\rm weak} (\omega),
		\end{align*}
		where the quasi-norm is given by
		\begin{align*}
			\|f\|_{L^{(\Phi,\phi)}_{\rm weak} (X, \omega)} = \sup_{B \in \mathfrak{B}} \|f\|_{\Phi, \phi, \omega, B, \text{weak}} =: \|f\|_{L^{(\Phi,\phi)}_{\rm weak} (\omega)}.
		\end{align*}
	\end{definition}


	Inspired by \cite[Proposition 3.2]{OM-Nakai2008}, we obtain the following lemma; its proof can be found in Appendix \ref{Appendix4}.

	\begin{lemma}\label{PPppw}
		Let \((X, \mathfrak{M}, \mu)\) be a measure space endowed with a ball-basis \(\mathfrak{B}\). Suppose \(\Phi, \Psi \in \mathscr{Y}\) are Young functions, \(\phi, \psi \in \mathscr{G}\), and \(\omega\) is a weight function. Then the following statements hold.
		
		\begin{list}{\rm (\theenumi)}{\usecounter{enumi}\leftmargin=1.2cm \labelwidth=1cm \itemsep=0.2cm \topsep=.2cm \renewcommand{\theenumi}{\arabic{enumi}}}
			\item\label{PPw} If \(\Phi(t) \le \Psi\left( C(\Phi, \Psi) t \right)\), then
			\[
			L^{(\Phi, \phi)} (\omega) \supset L^{(\Psi, \phi)} (\omega), \qquad \|f\|_{L^{(\Phi, \phi)}(\omega)} \le C(\Phi, \Psi) \, \|f\|_{L^{(\Psi, \phi)}(\omega)}.
			\]
			
			\item\label{ppw} If \(\psi(t) \le C(\psi, \phi) \,\phi(t)\), then
			\[
			L^{(\Psi, \phi)} (\omega) \supset L^{(\Psi, \psi)} (\omega), \qquad \|f\|_{L^{(\Psi, \phi)}(\omega)} \le \max\{1,\, C(\psi,\phi)\} \, \|f\|_{L^{(\Psi, \psi)}(\omega)}.
			\]
		\end{list}
	\end{lemma}

	\begin{definition}\label{Y-G condition}
		Let \(\Phi, \Psi \in \mathscr{Y}\) and \(\phi, \psi \in \mathscr{G}\). If the quadruple \((\Phi, \Psi, \phi, \psi)\) satisfies
		\[
		\Phi(t) \le \Psi\left(C(\Phi,\Psi)t\right) \quad \text{and} \quad \psi(t) \le C(\psi,\phi)\,\phi(t),
		\]
		then we say that \((\Phi,\Psi,\phi,\psi)\in\mathscr{Y} \otimes \mathscr{G}\).
	\end{definition}

	\begin{remark}
		Combining Lemma \ref{PPppw} with Definition \ref{Y-G condition} yields the inclusion
		\[
		L^{(\Phi, \phi)}(\omega) \supset L^{(\Psi, \psi)}(\omega),
		\]
		as well as the norm estimate
		\begin{align}\label{K}
			\|f\|_{L^{(\Phi, \phi)}(\omega)} \le C(\Phi, \Psi)\,\max\{1,\, C(\psi,\phi)\} \, \|f\|_{L^{(\Psi, \psi)}(\omega)} =: \mathscr{K} \|f\|_{L^{(\Psi, \psi)}(\omega)}.
		\end{align}
		We refer to Definition \ref{Y-G condition} as the {\tt \(\mathscr{Y} \otimes \mathscr{G}\)-condition}, which serves as an essential assumption in this paper.
	\end{remark}

	\subsection{Abstract weighted Campanato space endowed with a ball-basis}\label{Section3.3}~
	
%
%

	We recall that Campanato spaces on the underlying space \(\mathbb{R}^n\) serve to characterize the local oscillation of functions \cite{Nakai2002}, and their definition involves the mean value of functions. We first introduce the following notation. Let \((X, \mathfrak{M}, \mu)\) be a measure space endowed with a ball-basis \(\mathfrak{B}\), and let \(\omega \in A^\mathfrak{B}_p\). For a function \(f \in L_{\rm loc}^1(X)\) and a ball \(B \in \mathfrak{B}\), we set
	\[
	f_B := \frac{1}{\omega(B)} \int_B f(x) \, \omega(x) \, d\mu(x).
	\]
	
	Below we define the abstract weighted Campanato space endowed with a ball-basis.

	\begin{definition}
		Let \((X, \mathfrak{M}, \mu)\) be a measure space endowed with a ball-basis \(\mathfrak{B}\), let \(1 \le p < \infty\), and let \(\varphi : (0,\infty) \to (0,\infty)\) be a given function. For a weight \(\omega \in A^\mathfrak{B}_p\), the {\tt abstract weighted Campanato space endowed with a ball-basis} is defined as
		\begin{align*}
			\mathfrak{L}^p_\varphi (X, \omega) = \left\{ f \in L_{\rm loc}^p(X) : \|f\|_{\mathfrak{L}^p_\varphi (X, \omega)} < +\infty \right\} =: \mathfrak{L}^p_\varphi (\omega),
		\end{align*}
		where the norm is given by
		\begin{align}\label{Campanato norm}
			\|f\|_{\mathfrak{L}^p_\varphi (X, \omega)} = \sup_{B \in \mathfrak{B}} \frac{1}{\varphi(\omega(B))} \left( \frac{1}{\omega(B)} \int_B |f(x) - f_B|^p \, \omega(x) \, d \mu(x) \right)^{1/p} =: \|f\|_{\mathfrak{L}^p_\varphi (\omega)}.
		\end{align}
	\end{definition}

	\begin{remark}
		
		Consider \((\mathbb{R}^n, \mathscr{L}^n)\) with \(\omega(x) \equiv 1\). For \(1 \le p < \infty\) we restrict \(-n < \alpha < 1\). Choosing \(\varphi(r) = r^\alpha\) in the above definition of the abstract weighted Campanato space endowed with a ball-basis, the resulting space \(\mathfrak{L}^p_\varphi (\mathbb{R}^n)\) exhibits the following relationships with other classical spaces.
		
		\begin{itemize}
			
			\item If \(-n/p \le \alpha < 0\), then \(\mathfrak{L}^p_\varphi (\mathbb{R}^n) / \mathfrak{C}\) coincides with the space \(L^{p,\lambda}\), where \(\mathfrak{C}\) denotes the set of constant functions and \(\lambda = n + \alpha p\). In particular,
			\[\left\{
			\begin{array}{cl}
				\alpha = -n / p, & \mathfrak{L}^p_\varphi (\mathbb{R}^n) / \mathfrak{C} = L^p (\mathbb{R}^n), \\
				p = 1, & \mathfrak{L}^p_\varphi (\mathbb{R}^n) = {\rm BMO}_\varphi (\mathbb{R}^n). \text{ Furthermore, } \left\{
				\begin{array}{cl}
					\alpha = 0, & {\rm BMO}_\varphi = {\rm BMO} \supset L^\infty, \\
					0 < \alpha < 1, & {\rm BMO}_\varphi = {\rm Lip}_\alpha.
				\end{array}
				\right.
			\end{array}
			\right.\]
			
			
			\item Here \(\mathrm{Lip}_\alpha(\mathbb{R}^n)\) denotes the \(\alpha\)-H\"older space on \(\mathbb{R}^n\).
			
		\end{itemize}
		
	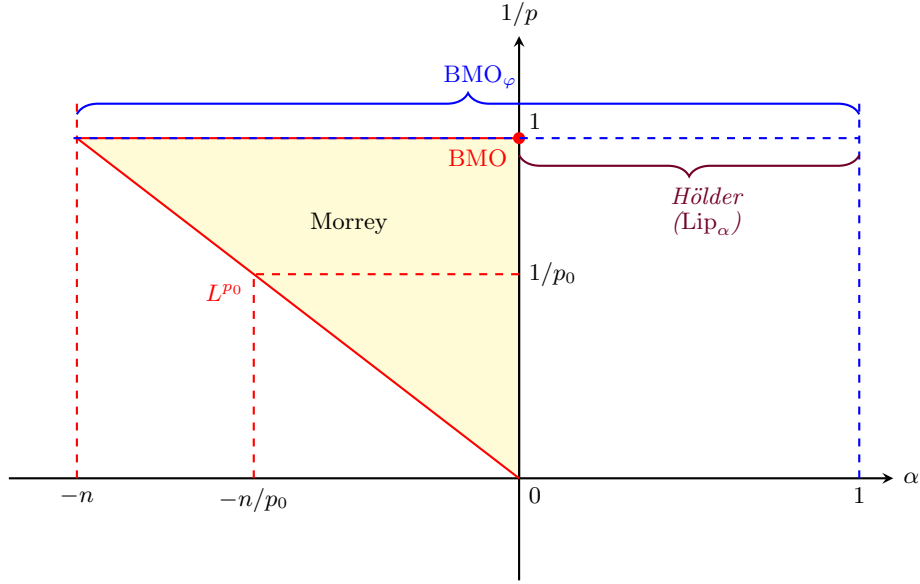
\begin{figure}[h!]
		\begin{center}
			\begin{tikzpicture}[scale=4.5, font=\small]
				\def\xmin{-1.5}
				\def\xmax{1.1}
				\def\ymin{-0.3}
				\def\ymax{1.3}
				
				\draw[-stealth, thick] (\xmin,0) -- (\xmax, 0) node[right] {\(\alpha\)};
				\draw[-stealth, thick] (0,\ymin) -- (0, \ymax) node[above] {\(1/p\)};
				
				\node at (0, 1) [above right] {\(1\)};
				\node at (1, 0) [below] {\(1\)};
				\node at (0, 0) [below right] {\(0\)};
				\node at (0, 0.6) [right] {\(1/p_0\)};
				\node at (-0.78, 0) [below] {\(-n/p_0\)};
				\node at (-1.3, 0) [below] {\(-n\)};
				
				\fill[yellow!20] (0, 0) -- (0, 1) -- (-1.3, 1) -- cycle;
				
				\node at (-0.5, 0.75) {\rm Morrey};
				
				\draw[thick] (0, 0) -- (0, 1);
				\draw[thick, red] (0, 1) -- (-1.3, 1);
				\draw[thick, red] (0, 0) -- (-1.3, 1);
				\node[red, below left] at (-0.78, 0.6) {\(L^{p_0}\)};
				\draw[thick, red, dashed] (-0.78, 0) -- (-0.78, 0.6);
				\draw[thick, red, dashed] (0, 0.6) -- (-0.78, 0.6);
				\draw[thick, red, dashed] (-1.3, 0) -- (-1.3, 1.12);
				
				\fill[red] (0, 1) circle (0.5pt) node[below left = 0.03cm] {\rm BMO};
				\draw[thick, blue, dashed] (-1.31, 1) -- (1, 1);
				\draw[decorate, decoration={brace, amplitude=8pt}, thick, blue] (-1.3, 1.07) -- (1, 1.07);
				\node[blue, right] at (-0.25, 1.18) {\({\rm BMO}_\varphi\)};
				\draw[thick, blue, dashed] (1, 0) -- (1, 1.12);
				
				\draw[decorate, decoration={brace, amplitude=8pt}, thick, purple!60!black] (1, 0.95) -- (0, 0.95);
				\node[purple!60!black] at (0.55, 0.83) {H\"older};
				\node[purple!60!black] at (0.55, 0.75) {(\({\rm Lip}_\alpha\))};
				
			\end{tikzpicture}
			\caption{Campanato spaces \(\mathfrak{L}^p_\varphi (\mathbb{R}^n, \mathscr{L}^n)\), \(\varphi(r) = r^\alpha\)}
		\end{center}
	\end{figure}
		 
	\end{remark}

	\section{Sparse domination of \(\Psi\)-BOOs}\label{Section4}
	
	\subsection{Some lemmas on ball-basis}\label{Section4.1}~
	
	We use the notation \(a \lesssim b\) to mean \(a \leq c\cdot b\) for some constant \(c>0\) independent of all relevant parameters, and we write \(a \sim b\) when both \(a \lesssim b\) and \(b \lesssim a\) hold.
	
	Throughout this work, \((X,\,\mathfrak{M},\,\mu)\) be a measure space endowed with a ball-basis \(\mathfrak{B}\). We shall frequently use several simple consequences of the axioms defining \(\mathfrak{B}\):
	
	\begin{list}{\rm (\theenumi)}{\usecounter{enumi}\leftmargin=1.2cm \labelwidth=1cm \itemsep=0.2cm \topsep=.2cm \renewcommand{\theenumi}{\arabic{enumi}}}
		\item\label{Ax-1} A set \(E \subset X\) is termed {\tt bounded} when there exists some ball \(B \in \mathfrak{B}\) containing \(E\).
		
		\item\label{Ax-2} For measurable sets \(E, F \subset X\), we write \(E \subset F\) a.s. (almost surely) when \(\mu(E \setminus F) = 0\).
		
		\item\label{Ax-3} A point \(x \in E\) is a {\tt density point} of a measurable set \(E \in \mathfrak{M}\) if for any \(\varepsilon > 0\) there exists a ball \(B \ni x\) such that \(\mu(B \cap E) > (1 - \varepsilon) \mu(B)\).
		
		\item\label{Ax-4} The {\tt density property} for a measure space \((X,\,\mathfrak{M},\,\mu)\) is characterized by the condition that, given any measurable set \(E\), almost every point of \(E\) is a density point.
		
		\item\label{Ax-5} A ball-basis \(\mathfrak{B}\) satisfies the {\tt doubling condition} if and only if there exists a constant \(\theta > 1\) such that, given any ball \(A \in \mathfrak{B}\) with \(A^\P \subsetneq X\), we can select a ball \(B \in \mathfrak{B}\) for which both inclusions \(A \subset B\) and the measure inequality \(\mu(B) \leq \theta \mu(A)\) hold.
		
	\end{list}
	
	The geometric properties of the ball-basis discussed below can be found in \cite{Abstract2019, Caomingming2023, Abstract2021, Abstract2023, SZ2026}.

	\begin{lemma}\label{Lm 3.1}
		Let \((X,\,\mathfrak{M},\,\mu)\) be a measure space endowed with a ball-basis \(\mathfrak{B}\). Assume that there exist a ball \(B \in \mathfrak{B}\) and a sequence of balls \(\{G_k\}_{k \geq 1} \subset \mathfrak{B}\) such that \(G_k \cap B \neq \varnothing\) for each \(k\) and \(\lim_{k \to \infty} \mu(G_k) = r := \sup_{A \in \mathfrak{B}} \mu(A)\). Then \(X \subset \bigcup_k G_k^\P\). Moreover, for any ball \(A \in \mathfrak{B}\), there exists an integer \(k_0\) such that \(A \subset G_k\) whenever \(k \ge k_0\).
	\end{lemma}
	
	\begin{lemma}\label{Lm more}
		Let \((X,\,\mathfrak{M},\,\mu)\) be a measure space endowed with a ball-basis \(\mathfrak{B}\) that satisfies Proposition \ref{P-ball-basis}. The following statements then hold:
		\begin{list}{\rm (\theenumi)}{\usecounter{enumi}\leftmargin=1.2cm \labelwidth=1cm \itemsep=0.2cm \topsep=.2cm \renewcommand{\theenumi}{\arabic{enumi}}}
			\item\label{L-B1} If a set \(E \subset X\) is bounded and covered by a family of balls \(\mathcal{G} \subset \mathfrak{B}\) (that is, \(E \subset \bigcup_{G \in \mathcal{G}} G\)), then one can extract a finite or infinite pairwise disjoint subfamily \(\{G_k\}_{k \in I} \subset \mathcal{G}\), where the index set \(I\) is either finite or \(I = \mathbb{N}_+\), such that \(E \subset \bigcup_{k \in I} G_k^\P\).
			
			\item\label{L-B2} Under the additional assumption that \(\mathfrak{B}\) satisfies the density property, the following holds: for every bounded measurable set \(E\) with \(\mu(E)>0\) and every \(\varepsilon > 0\), there exists a sequence \(\{B_k\} \subset \mathfrak{B}\) satisfying
			\begin{align*}
				\mu \left( \bigcup_k B_k \setminus E \right) < \varepsilon \quad \text{and} \quad \mu \left( E \setminus \bigcup_k B_k \right) < \alpha \mu(E),
			\end{align*}
			with some admissible constant \(\alpha \in (0,1)\).
			
			\item\label{L-B3} If, furthermore, \(\mathfrak{B}\) satisfies the density property, then for every bounded measurable set \(E \subset X\) one can find a sequence of balls \(\{B_k\}_{k \in \mathbb{N}_+} \subset \mathfrak{B}\) such that
			\begin{align*}
				E \subset \bigcup_k B_k \; a.s. \quad \text{and} \quad \sum_k \mu(B_k) \leq 2\mathscr{C}_0\,\mu(E).
			\end{align*}
			
			\item\label{L-B4} Let \(A \in \mathfrak{B}\) and let \(\mathcal{G} \subset \mathfrak{B}\) be a pairwise disjoint family of balls such that for every \(G \in \mathcal{G}\), and there exist constants \(c_1,\, c_2 > 0\) satisfying
			\begin{align*}
				G^\P \cap A \neq \varnothing \quad \text{and} \quad 0 < c_1 \leq \mu(G) \leq c_2 <\infty.
			\end{align*}
			Then \(\mathcal{G}\) is finite, and its cardinality \(\#\mathcal{G}\) satisfies
			\begin{align*}
				\#\mathcal{G} \lesssim c_1^{-1} \max\{c_2,\; \mu(A)\}.
			\end{align*}
			(Here \(\#\mathcal{G}\) denotes the number of elements in the family \(\mathcal{G}\).)
			
			\item\label{L-B5} The condition \eqref{list:BB3} holds if and only if the density property is satisfied.
			
		\end{list}
	\end{lemma}
	
	\subsection{Properties of \(\Psi\)-BOOs}\label{Section4.2}~
	
	
	The proofs of the lemmas in this subsection are taken from \cite{SZ2026} and will not be repeated here.
	
	For a \(\mathbb{B}\)-valued operator \(T\) satisfying \(Tf(x) = \|\mathscr{T}f(x)\|_{\mathbb{B}}\), we consider condition \eqref{BOO-1}. For any \(A, B \in \mathfrak{B}\) with \(A \subset B\) we introduce  
	\begin{align*}
		\Delta_T(A,B) := \sup_{f \in L^{(\Psi,\psi)}(X)} \sup_{x\in A} \frac{ \left\| \mathscr{T}(f\,\mathbf{1}_{B^\P})(x) - \mathscr{T}(f \, \mathbf{1}_{A^\P})(x) \right\|_{\mathbb{B}}} {\|f\|_{\Phi,\phi,B^\P}}.
	\end{align*}
	
	\begin{lemma}\label{Lm 3.3}
		Let \((X,\,\mathfrak{M},\,\mu)\) be a measure space endowed with a ball-basis \(\mathfrak{B}\). Given a Young function \(\Phi \in \mathscr{Y}\) and a function \(\phi \in \mathscr{G}\), the following hold:
		\begin{list}{\rm (\theenumi)}{\usecounter{enumi}\leftmargin=1.2cm \labelwidth=1cm \itemsep=0.2cm \topsep=.2cm \renewcommand{\theenumi}{\alph{enumi}}}
		\item\label{Lm 3.3-a} If \(A, B, C \in \mathfrak{B}\) satisfy \(A \subset B \subset C\), then \(\Delta_T(A,B) \lesssim \Delta_T(A, C)\).
		
		\item\label{Lm 3.3-b} For any \(A, B \in \mathfrak{B}\) with \(A \subset B\), we have
		\begin{align*}
			\langle \|f \, \mathbf{1}_{B^\P}\|\rangle_{\Phi,\phi,A} \lesssim \frac{\mu(B) \phi(\mu(B))}{\mu(A)\phi(\mu(A))} \|f\|_{\Phi,\phi, B^\P}.
		\end{align*}
		\end{list}
	\end{lemma}

	\begin{lemma}\label{Lm 3.4}
		Let \((X,\,\mathfrak{M},\,\mu)\) be a measure space endowed with a ball-basis \(\mathfrak{B}\), and let \((\Psi,\psi) \in E_{mb}\). Assume that \(T\) is a \(\mathbb{B}\)-valued linear operator satisfying the condition \eqref{BOO-2} and that \(T \in \mathbb{W}_{\Psi, \psi}\). Then for every \(f \in L^{(\Psi, \psi)}(X)\) the following estimates hold:
		
		\begin{list}{\rm (\theenumi)}{\usecounter{enumi}\leftmargin=1.2cm \labelwidth=1cm \itemsep=0.2cm \topsep=.2cm \renewcommand{\theenumi}{\alph{enumi}}}
			\item\label{Lm 3.4-a} For any balls \(A,B \in \mathfrak{B}\) with \(A \subset B\),
			\begin{align*}
				\Delta_T(A,B)\lesssim\left(\mathscr{C}_2(T)+\|T\|_{L^{(\Psi, \psi)}\to wL^{(\Psi,\psi)}}\right)\frac{\mu(B)\phi(\mu(B)) }{\mu(A)\phi(\mu(A))}.
			\end{align*}
			
			\item\label{Lm 3.4-b} For any balls \(A, B, C \in \mathfrak{B}\) with \(A \subset B \subset C\),
			\begin{align*}
				\Delta_T(A,C)\lesssim\left(\mathscr{C}_2(T)+\|T\|_{L^{(\Psi, \psi)}\to wL^{(\Psi,\psi)}}+\Delta_T(A,B)\right) \frac{\mu(C)\phi(\mu(C)) }{\mu(B)\phi(\mu(B))}.
			\end{align*}
			
			\item\label{Lm 3.4-c} For any balls \(A,B \in \mathfrak{B}\) with \(A \subset B\),
			\begin{align*}
				\Delta_T(A,B^{[k]}) \lesssim \left(\mathscr{C}_2(T) + \|T\|_{ L^{(\Psi, \psi)}\to wL^{(\Psi,\psi)}}+\Delta_T(A,B)\right) \frac{\mathscr{C}_0^k \phi \left(\mathscr{C}_0^k\mu(B)\right)}{\phi(\mu(B))}, \quad k\geq 1.
			\end{align*}
			
			\item\label{Lm 3.4-d} If \(\mathfrak{B}\) also satisfies the doubling condition, then \(T\) satisfies condition \eqref{BOO-1}. Consequently, \(T\) is a \(\mathbb{B}\)-valued \(\Psi\)-BOO.
			
		\end{list}
	\end{lemma}

	\begin{lemma}\label{Lm 3.5}
		Let \((X,\,\mathfrak{M},\,\mu)\) be a measure space endowed with a ball-basis \(\mathfrak{B}\). Assume that \(T\) is a \(\mathbb{B}\)-valued \(\Psi\)-BOO and that \(T \in \mathbb{W}_{\Psi, \psi}\). Define
		\begin{align*}
			\mathscr{C}(T) := \mathscr{C}_1(T) + \mathscr{C}_2(T) + \|T\|_{ L^{(\Psi, \psi)}(X) \to wL^{(\Psi,\psi)}(X)}.
		\end{align*}
		Then the following statements hold.
		
		\begin{list}{\rm (\theenumi)}{\usecounter{enumi}\leftmargin=1.2cm \labelwidth=1cm \itemsep=0.2cm \topsep=.2cm \renewcommand{\theenumi}{\arabic{enumi}}}
			\item\label{Lm 3.5-1} For every ball \(B \in \mathfrak{B}\) there exists a ball \(\widetilde{B} \in \mathfrak{B}\) such that
			\begin{align*}
				B^{[2]} \subset \widetilde{B}, \quad \Delta_T(B^{[2]}, \widetilde{B}) \lesssim \mathscr{C}(T), \quad \text{and either} \; \widetilde{B}^{[1]}=\widetilde{B} \; \text{or} \; \mu(\widetilde{B}) \geq 2 \mu(B).
			\end{align*}
			
			\item\label{Lm 3.5-2} For every ball \(B \in \mathfrak{B}\) satisfying \(B^\P = B\), there exists \(\widetilde{B} \in \mathfrak{B}\) such that
			\begin{align*}
				B^{[2]} \subset \widetilde{B}, \quad \Delta_T(B^{[2]}, \widetilde{B}) \lesssim \mathscr{C}(T), \quad \text{and} \quad \mu(\widetilde{B}) \geq 2 \mu(B).
			\end{align*}
			
			\item\label{Lm 3.5-3} For every ball \(B \in \mathfrak{B}\) there exists a sequence \(\{B_k\}_{k \geq 0} \subset \mathfrak{B}\) with \(B_0 = B\) such that for all \(k \geq 0\),
			\begin{align*}
				X = \bigcup_k B_k, \quad B_k^{[2]} \subset B_{k+1}, \quad \text{and} \quad \Delta_T(B_k^{[2]}, B_{k+1}) \lesssim \mathscr{C}(T).
			\end{align*}
			
		\end{list}
	\end{lemma}

	\begin{lemma}\label{Lm 3.6}
		Let \((X,\,\mathfrak{M},\,\mu)\) be a measure space endowed with a ball-basis \(\mathfrak{B}\), and assume that \(T\) is a \(\Psi\)-BOO with \(T \in \mathbb{W}_{\Psi, \psi}\). Set \(\lambda \geq 3 \mathscr{C}_0^4\), let \(F \subset X\) be a measurable set, and choose a ball \(A \in \mathfrak{B}\) such that
		\begin{align*}
			F \cap A \ne \varnothing \quad \text{and} \quad \mu(F) \leq \lambda^{-1}\mu(A).
		\end{align*} 
		Then there exists a family \(\mathcal{G} \subset \mathfrak{B}\) satisfying the following properties:
		\begin{list}{\rm (\theenumi)}{\usecounter{enumi}\leftmargin=1.2cm \labelwidth=1cm \itemsep=0.2cm \topsep=.2cm \renewcommand{\theenumi}{\arabic{enumi}}}
			\item\label{Lm 3.6-1} For every \(G \in \mathcal{G}\), \(\displaystyle F \cap A^\P \cap G \ne \varnothing\),
			
			\vspace{6pt}
			
			\item\label{Lm 3.6-2} \(\displaystyle F \cap A^\P \subset \bigcup_{G \in \mathcal{G}} G \; a.s.\),
			
			\vspace{6pt}
			
			\item\label{Lm 3.6-3} \(\displaystyle \mu\left(\bigcup_{G \in \mathcal{G}} G^\P\right) \leq 3\lambda^{-1}\mathscr{C}_0^2 \mu(A)\),
			
			\vspace{6pt}
			
			\item\label{Lm 3.6-4} For each \(G \in \mathcal{G}\) one can find a ball \(\widetilde{G} \in \mathfrak{B}\) such that  
			\begin{align*}
				\widetilde{G} \not\subset F, \quad G^{[2]} \subset \widetilde{G} \subset A^\P, \quad \text{and} \quad \Delta_T(G^{[2]},\widetilde{G}) \lesssim \mathscr{C}(T),
			\end{align*}
			with constants independent of \(F\), \(A\), \(\lambda\) and \(\mathcal{G}\).
		\end{list}
	\end{lemma}

	Let \(T\) be a \(\mathbb{B}\)-valued operator such that \(Tf(x) = \|\mathscr{T}f(x)\|_{\mathbb{B}}\). The truncation operator is given by  
	\begin{align*}
		T^*f(x) := \sup_{B \in \mathfrak{B} : x \in B} \left\| \mathscr{T}f(x) - \mathscr{T}(f \, \mathbf{1}_{B^\P})(x) \right\|_{\mathbb{B}}.
	\end{align*}

	\begin{theorem}\label{Th 3.7}
		Let \((X,\,\mathfrak{M},\,\mu)\) be a measure space endowed with a ball-basis \(\mathfrak{B}\), and let \((\Psi,\psi) \in E_{mb}\). Assume that \(T\) is a \(\mathbb{B}\)-valued linear operator satisfying the condition \eqref{BOO-2} and that \(T \in \mathbb{W}_{\Psi, \psi}\). Then for every \(f \in L^{(\Psi, \psi)}(X)\), we have
		\begin{align*}
			\|T^*\|_{L^{(\Psi, \psi)}(X) \to wL^{(\Psi, \psi)}(X)} \lesssim \mathscr{C}_2(T) + \|T\|_{L^{(\Psi, \psi)}(X) \to wL^{(\Psi, \psi)}(X)}.
		\end{align*}
	\end{theorem}

	Moreover, for the subsequent analysis, we need to introduce the following operator:
	\begin{align*}
		\Upsilon f(x) := \max\left\{|Tf(x)|, \, T^*f(x), \, \mathscr{C}(T) \mathcal{M}_{\mathfrak{B}, \Phi, \phi} f(x) \right\}.
	\end{align*}
	Hence, it is immediate that
	\begin{align}\label{3.1}
		\|\Upsilon\| := \|\Upsilon\|_{L^{ (\Psi, \psi)} \to wL^{(\Psi, \psi)}} \lesssim \mathscr{C}(T)
	\end{align}
	For a measure space \((X,\,\mathfrak{M},\,\mu)\) endowed with a ball-basis \(\mathfrak{B}\), we fix in the sequel a ball \(B_0 \in \mathfrak{B}\) satisfying
	\begin{align}\label{3.2}
		{\rm supp}\,f \subset B_0^{[3]} , 
	\end{align}
	which entails no loss of generality for our analysis.

	\begin{lemma}\label{Lm 3.8}
		Let \((X,\,\mathfrak{M},\,\mu)\) be a measure space endowed with a ball-basis \(\mathfrak{B}\). Let \(\Psi,\Phi\in\mathscr{Y}\) and \(\psi, \phi \in \mathscr{G}\). Suppose \(T\) is a \(\mathbb{B}\)-valued \(\Psi\)-BOO and \(T \in \mathbb{W}_{\Psi, \psi}\). Take \(\lambda \geq 3 \mathscr{C}_0^4\). For the ball \(B_0\) in \eqref{3.2}, there exists a family \(\mathcal{G} = \mathcal{G}(B_0) \subset \mathfrak{B}\) satisfying the following conditions:
		
		\begin{list}{\rm (\theenumi)}{\usecounter{enumi}\leftmargin=1.2cm \labelwidth=1cm \itemsep=0.2cm \topsep=.2cm \renewcommand{\theenumi}{\alph{enumi}}}
			\item\label{Lm 3.8-a} For any \(A \subset \mathcal{G}\), there exists \(\mathcal{F}(A) \subset \mathcal{G}\) such that:
			
			\begin{list}{\rm (\theenumi\theenumii)}{\usecounter{enumii}\leftmargin=.4cm \labelwidth=.8cm\itemsep=0.2cm\topsep=.1cm\renewcommand{\theenumii}{-\arabic{enumii}}}
				\item\label{Lm 3.8-a1} \(A^\P \cap B \ne \varnothing\) for all \(B \in \mathcal{F}(A)\),
				
				\vspace{5pt}
				
				\item\label{Lm 3.8-a2} \(\mu\left(\bigcup_{B \in \mathcal{F}(A)}B^\P\right) \leq 3 \lambda^{-1} \mathscr{C}_0^2 \mu(A)\),
				
				\vspace{5pt}
				
				\item\label{Lm 3.8-a3} \(\Upsilon(f \, \mathbf{1}_{A^{[3]}})(x) \lesssim \mathscr{C}(T) \Psi^{-1}(\lambda \phi(\mu(A))) \|f\|_{\Phi, \phi, A^{[3]}}, \quad a.e. \; x \in A^\P \setminus \bigcup_{B \in \mathcal{F}(A)}B\).
			\end{list}
			
			\vspace{5pt}
			
			\item\label{Lm 3.8-b} For any \(B \in \mathcal{F}(A)\), there exist \(\widetilde{B} \in \mathfrak{B}\) and \(\xi \in \widetilde{B}\), which yields the following result:
			\begin{list}{\rm (\theenumi\theenumii)}{\usecounter{enumii}\leftmargin=.4cm \labelwidth=.8cm\itemsep=0.2cm\topsep=.1cm\renewcommand{\theenumii}{-\arabic{enumii}}}
				
				\vspace{5pt}
				
				\item\label{Lm 3.8-b1} \(B^{[2]} \subset \widetilde{B} \subset A^\P\),
				
				\vspace{5pt}
				
				\item\label{Lm 3.8-b2} \(\Upsilon(f \, \mathbf{1}_{A^{[3]}})(\xi) \lesssim \mathscr{C}(T) \Psi^{-1}(\lambda \phi(\mu(A))) \|f\|_{\Phi, \phi, A^{[3]}}\),
				
				\vspace{5pt}
				
				\item\label{Lm 3.8-b3} \(\|\mathscr{T}(f\,\mathbf{1}_{\widetilde{B}^ \P})(x) - \mathscr{T}(f\,\mathbf{1}_{B^{[3]}})(x)\|_{\mathbb{B}} \lesssim \mathscr{C}(T) \Psi^{-1}(\lambda \phi(\mu(A))) \|f\|_{ \Phi,\phi,A^{[3]}}, \quad\text{for all} \; x \in B^{[2]}\).
			\end{list}
		\end{list}
	\end{lemma}

	For the convenience of subsequent exposition, we introduce the following notation:
	
	\begin{itemize}
		\item The set \(\mathcal{F}(A)\) in Lemma \ref{Lm 3.8} denotes the stopping time balls associated with \(A\).
		
		\item The family \(\mathcal{F}_k(A)\) denotes the collection of the \(k\)-th generation stopping time balls associated with \(A\). It is defined recursively by
		\begin{align*}
			\mathcal{F}_0(B_0) := \{B_0\}, \quad \mathcal{F}_1(B_0) := \mathcal{F}(B_0), \quad \text{and} \quad \mathcal{F}_{k+1}(B_0) := \bigcup_{B \in \mathcal{F}_k(B_0)} \mathcal{F}(B), \;\; k \geq 1.
		\end{align*}
		
		\item The set \(\mathcal{G}\) denotes the collection of all generation balls, given by
		\begin{align*}
			\mathcal{G} = \mathcal{G}(B_0) := \bigcup_{k \geq 0} \mathcal{F}_k (B_0).
		\end{align*}
	\end{itemize}
	Given \(A, B \in \mathcal{G}\) and \(k \geq 0\), if \(B \in \mathcal{F}_k(A)\), we write \(\mathcal{F}^k(B) = A\). In which case we say that \(A\) is the \(k\)-th generation ancestor of \(B\), or that \(B\) is one of the \(k\)-th generation descendants of \(A\). Repeated application of Lemma \ref{Lm 3.8} part \eqref{Lm 3.8-a2} yields
	\begin{align}\label{3.33}
		\mu\left(\bigcup_{B \in \mathcal{F}_k(A)}B \right) \leq (3 \lambda^{-1} \mathscr{C}_0^2)^k \mu(A), \quad \text{for all} \; k \geq 1.
	\end{align}

	\begin{lemma}\label{Lm 3.9}
		Let \((X,\,\mathfrak{M},\,\mu)\) be a measure space endowed with a ball-basis \(\mathfrak{B}\) and \(\lambda > 3\mathscr{C}_0^6\). Fix \(B_0 \in \mathfrak{B}\) satisfying \eqref{3.2}. Then the following properties hold:
		\begin{list}{\rm (\theenumi)}{\usecounter{enumi}\leftmargin=1.2cm \labelwidth=1cm \itemsep=0.2cm \topsep=.2cm \renewcommand{\theenumi}{\alph{enumi}}}
			\item\label{Lm 3.9-1} When \(\lambda\) is sufficiently large, there exist two types of sparse families within the collection \(\mathcal{G}\) from Lemma \ref{Lm 3.8}, namely: a \(2^{-1}\)-sparse family and a \(2^{-1}\mathscr{C}_0^{-3}\)-sparse family.
			
			\item\label{Lm 3.9-2} For any ball \(B\) belonging to the \(2^{-1}\)-sparse family, there exists a sequence \(\{B_j\}_{j = 0}^k\) that is a subset of the \(2^{-1}\)-sparse family, such that \(B_{j+1} \in \mathcal{F}(B_j)\) for \(j = 0, 1, ..., k-1\) and \(B_k = B\). There exist balls \(\widetilde{B}_j\) and points \(\xi_j \in \widetilde{B}_{j+1}\) such that the following estimates hold:
			\begin{align*}
				&B_{j + 1}^{[2]} \subset \widetilde{B}_{j + 1} \subset B_j^\P,  \\
				&\Upsilon(f \, \mathbf{1}_{B_j^{[3]}})(\xi_j) \lesssim \mathscr{C}(T) \Psi^{-1}(\lambda \phi(\mu(B_j))) \|f\|_{\Phi, \phi, B_j^{[3]}},  \\
				&\left\| \mathscr{T}(f \, \mathbf{1}_{\widetilde{B}_{j + 1}^\P})(x) - \mathscr{T}(f \, \mathbf{1}_{B_{j + 1}^{[3]}})(x) \right\|_{\mathbb{B}} \lesssim \mathscr{C}(T) \Psi^{-1}(\lambda \phi(\mu(B_j))) \|f\|_{ \Phi, \phi, B_j^{[3]}}, \;\; \text{for all} \; x \in B_0.
			\end{align*}
		\end{list}
	\end{lemma}

	\subsection{Proof of Theorem \ref{main}}\label{Section4.3}~

	In this subsection, we aim to prove Theorem \ref{main}.
	
	\begin{proof}[Proof of Theorem \ref{main}]
		
		In \cite{SZ2026}, we proved \eqref{ine main1}. Here we prove \eqref{ine main2}. For any \(B_0 \in \mathfrak{B}\) satisfying \eqref{3.2}, the estimate \eqref{ine main2} is equivalent to
		\begin{align}\label{3.4}
			\left\| [\mathscr{T}, b] (f \, \mathbf{1}_{B_0^{[3]}}) (x) \right\|_\mathbb{B} \lesssim \mathscr{C} (T) \left[ \mathcal{A}^b_{\mathcal{S}_1, \Phi,\phi} f(x) + \mathcal{A}^b_{\mathcal{S}_2, \Phi,\phi} f(x) \right], \quad a.e. \; x \in B,
		\end{align}
		where both \(\mathcal{S}_1\) and \(\mathcal{S}_2\) are \(\frac{1}{2\mathscr{C}_0^3}\)-sparse. For the above \(B_0\) and for \(\lambda > 3\mathscr{C}_0^6\), Lemma \ref{Lm 3.9} part \eqref{Lm 3.9-1} yields that \(\mathcal{S}_1\) and \(\mathcal{S}_2\) are \(\frac{1}{2\mathscr{C}_0^3}\)-sparse and that there exists a family \(\mathcal{S}'\) which is \(\frac{1}{2}\)-sparse. Moreover, Lemma \ref{Lm 3.9} part \eqref{Lm 3.9-2} implies that for every ball in \(\mathcal{S}'\), one can find a chain \(\{B_j\}_{j=0}^k \subset \mathcal{S}'\) such that
		\begin{align*}
			B_{j+1} \in \mathcal{F}(B_j), \; B_k = B, \; \text{for } j = 0, 1, ..., k-1.
		\end{align*}
		and \(B_{j+1}^{[2]} \subset \widetilde{B}_{j+1} \subset B_j^\P\). The following lemma was proved in \cite{SZ2026}.
		
		\begin{lemma}\label{Lm 3.10}
			For any \(B \in \mathcal{S}'\), define \(r(B) = [\log_{\mathscr{C}_0^2} \mu(B)]\). Then there exists a set \(F_0\) of measure zero such that the following estimate holds:
			\begin{align}\label{3.5}
				\Upsilon(f \, \mathbf{1}_{B^{[3]}})(x) \lesssim \mathscr{C}(T) \Psi^{-1}(\lambda \phi(\mu(B))) \|f\|_{\Phi, \phi, B^{[3]}}, \quad x \in \left(B^\P \setminus \bigcup_{G \in \mathcal{S}': r(G) < r(B)} G^\P\right) \setminus F_0.
			\end{align}
		\end{lemma}
		\begin{lemma}\label{Lm 3.11}
			For all \(x \in B_k\) and for \(j = 0, 1, ..., k - 1\), we have
			\begin{align}\label{3.6}
				\|\mathscr{T}(f \, \mathbf{1}_{B_j^{[3]}})(x) - \mathscr{T} (f \, \mathbf{1}_{B_{j + 1}^{[3]}}) (x)\|_\mathbb{B} \lesssim \mathscr{C} (T) \|f\|_{\Phi, \phi, B_j^{[3]}}.
			\end{align}
		\end{lemma}
		
		Notice that \([T, b] g(x) = b(x) Tg(x) - T(bg)(x)\). Taking \(c = b_{B_j^{[3]}}\), we obtain
		\begin{align}\label{3.7}
			[T, b] g(x) = (b(x) - c) Tg(x) + T((c - b) g)(x).
		\end{align}
		
		\textbf{On the one hand}, set
		\begin{align}\label{3.8}
			g = f \mathbf{1}_{B_j^{[3]}} - f \mathbf{1}_{B_{j + 1}^{[3]}},
		\end{align}
		Then, together with the \((\Phi, \phi) \in E_{mb}\) condition, we obtain
		\begin{align}\label{3.9}
			|T g(x)| \leq |T (f \mathbf{1}_{B_j^{[3]}}) (x)| \lesssim \Upsilon(f \, \mathbf{1}_{B_j^{[3]}})(x) \lesssim \mathscr{C}(T) \|f\|_{\Phi, \phi, B_j^{[3]}}
		\end{align}
		and
		\begin{align}\label{3.10}
			|T((c - b) g)(x)| \leq |T ((c - b) f \mathbf{1}_{B_j^{[3]}}) (x)| \lesssim \Upsilon((c - b) f \, \mathbf{1}_{B_j^{[3]}})(x) \lesssim \mathscr{C}(T) \|(c - b) f\|_{\Phi, \phi, B_j^{[3]}}.
		\end{align}
		From \eqref{3.7} through \eqref{3.10}, we deduce
		\begin{align}\label{3.11}
			\left\|[\mathscr{T}, b] \left(f \, \mathbf{1}_{B_j^{[3]}} \right)(x) - [\mathscr{T}, b] \left(f \, \mathbf{1}_{B_{j + 1}^{[3]}} \right) (x) \right\|_\mathbb{B} \lesssim \mathscr{C} (T) \left( |b - c| \|f\|_{\Phi, \phi, B_j^{[3]}} + \|(c - b) f\|_{\Phi, \phi, B_j^{[3]}} \right).
		\end{align}
		
		\textbf{On the other hand}, from \eqref{3.5}, \eqref{3.7} and the \((\Phi, \phi) \in E_{mb}\) condition, we obtain
		\begin{align}\label{3.12}
			\|[\mathscr{T}, b](f & \mathbf{1}_{B^{[3]}} ) (x) \|_\mathbb{B} \leq (b(x) - c) \| T (f \mathbf{1}_{B^{[3]}} )(x) \|_\mathbb{B} + \|T ((c - b) (f \mathbf{1}_{B^{[3]}} ) )(x) \|_\mathbb{B} \nonumber \\
			& \lesssim \mathscr{C} (T) ( |b - c| \|f\|_{\Phi, \phi, B^{[3]}} + \|(c - b) f\|_{\Phi, \phi, B^{[3]}} ).
		\end{align}
		
		Consequently, from \eqref{3.11} and \eqref{3.12} we deduce
		\begin{align*}
			|[T, b] & f(x)| = \left\|[\mathscr{T}, b] \left(f \mathbf{1}_{B_0^{[3]}} \right)(x) \right\|_\mathbb{B} \\
			& \leq \sum_{j = 0}^{k - 1} \left\|[\mathscr{T}, b] \left(f \mathbf{1}_{B_j^{[3]}} \right)(x) - [\mathscr{T}, b] \left(f \mathbf{1}_{B_{j+1}^{[3]}} \right)(x) \right\|_\mathbb{B} + \left\|[\mathscr{T}, b] \left(f \mathbf{1}_{B_k^{[3]}} \right)(x) \right\|_\mathbb{B} \\
			& \lesssim \mathscr{C} (T) \sum_{j = 0}^k \left( |b - b_{B_j^{[3]}}| \|f\|_{\Phi, \phi, B_j^{[3]}} + \|(b_{B_j^{[3]}} - b) f\|_{\Phi, \phi, B_j^{[3]}} \right) \\
			& \lesssim \mathscr{C} (T) \sum_{B \in \mathcal{S}_1 \cup \mathcal{S}_2} \left( |b - b_{B^{[3]}}| \|f\|_{\Phi, \phi, B^{[3]}} + \|(b_{B^{[3]}} - b) f\|_{\Phi, \phi, B^{[3]}} \right) \cdot \mathbf{1}_{B_0^{[3]}} \\
			& \leq \mathscr{C} (T) \left[ \mathcal{A}^b_{\mathcal{S}_1, \Phi,\phi} f(x) + \mathcal{A}^b_{\mathcal{S}_2, \Phi,\phi} f(x) \right].
		\end{align*}
		
		Thus we obtain \eqref{3.4}, which completes the proof of Theorem \ref{main}.
	\end{proof}

	\subsection{Proof of Theorem \ref{main'}}\label{Section4.4}~
	
	
	In this subsection, our main task is to prove Theorem \ref{main'}. To this end, we first need the following lemma.
	
	\begin{lemma}\label{G-Holder}
		Let \((X, \mathfrak{M}, \mu)\) be a measure space endowed with a ball-basis \(\mathfrak{B}\), let \(\omega \in A^\mathfrak{B}_p\), and fix a ball \(B \in \mathfrak{B}\). Suppose \(\Phi_i \in \mathscr{Y}\) and \(\phi_i \in \mathscr{G}\) for \(i = 1, 2, 3\). Assume there exists a constant \(c > 0\) such that for all \(r, s > 0\) we have
		\begin{align*}
			\Phi_2^{-1} (r \phi_2(s)) \Phi_3^{-1} (r \phi_3(s)) \leq c \Phi_1^{-1} (r \phi_1(s)).
		\end{align*}
		If \(\|f\|_{\Phi_2, \phi_2, \omega, B} < \infty\) and \(\|g\|_{\Phi_3, \phi_3, \omega, B} < \infty\), then \(\|fg\|_{\Phi_1, \phi_1, \omega, B} < \infty\).
	\end{lemma}

	\begin{proof}
		We follow the proof strategy of \cite[Theorem 4.1]{OM-Nakai2008}. Without loss of generality we may assume \(\|f\|_{\Phi_2, \phi_2, \omega, B} = \|g\|_{\Phi_3, \phi_3, \omega, B} = 1\). For almost every \(x \in B\), define
		\begin{align*}
			r = \max \left( \frac{\Phi_2 (|f(x)|)}{\phi_2 (\omega(B))}, \frac{\Phi_3 (|g(x)|)}{\phi_3 (\omega(B))} \right) < \infty.
		\end{align*}
		This implies
		\begin{align*}
			\Phi_2 (|f(x)|) & \leq r \phi_2 (\omega(B)) \Longrightarrow |f(x)| \leq \Phi^{-1}_2 (r \phi_2 (\omega(B))), \\
			\Phi_3 (|g(x)|) & \leq r \phi_3 (\omega(B)) \Longrightarrow |g(x)| \leq \Phi^{-1}_3 (r \phi_3 (\omega(B))).
		\end{align*}
		Consequently,
		\begin{align*}
			|f(x)g(x)| \leq \Phi_2^{-1} (r \phi_2(\omega(B))) \Phi_3^{-1} (r \phi_3 (\omega(B))) \leq c \, \Phi_1^{-1}(r \phi_1 (\omega(B))),
		\end{align*}
		that is,
		\begin{align}\label{3.13}
			\Phi_1 \left(\frac{|f(x)g(x)|}{2c}\right) \leq \frac{1}{2} \Phi_1 \left(\frac{|f(x)g(x)|}{c}\right) \leq \frac{1}{2} \left(\frac{\Phi_2 (|f(x)|)}{\phi_2 (\omega(B))} + \frac{\Phi_3 (|g(x)|)}{\phi_3 (\omega(B))}\right) \phi_1 (\omega(B)).
		\end{align}
		On the other hand, we have
		\begin{equation}
			\begin{aligned}
				\int_B \Phi_2 (|f(x)|) \, \omega \, d \mu \leq \omega(B) \phi_2 (\omega(B)), \\
				\int_B \Phi_3 (|f(x)|) \, \omega \, d \mu \leq \omega(B) \phi_3 (\omega(B)).
			\end{aligned}
			\label{3.14}
		\end{equation}
		Integrating \eqref{3.13} over \(B\) and using \eqref{3.14} we obtain
		\begin{align*}
			\int_B \Phi_1 \left(\frac{|f(x)g(x)|}{2c}\right) \, \omega \, d\mu \leq \omega(B) \phi_1 (\omega(B)),
		\end{align*}
		which yields
		\begin{align*}
			\|fg\|_{\Phi_1, \phi_1, \omega, B} \leq 2c.
		\end{align*}
		Therefore,
		\begin{align}\label{3.15}
			\|fg\|_{\Phi_1, \phi_1, \omega, B} \leq 2c\|f\|_{\Phi_2, \phi_2, \omega, B} \|g\|_{\Phi_3, \phi_3, \omega, B}.
		\end{align}
	\end{proof}

	We now turn to the proof of Theorem \ref{main'}.
	
	\begin{proof}[Proof of Theorem \ref{main'}]
		We observe that for any \(x \in B\),
		\begin{align}\label{3.16}
			\mathcal{A}_{\mathcal{S}, \Phi, \phi} f(x) = \sum_{B \in \mathcal{S}} \|f\|_{\Phi,\phi,B} \cdot \mathbf{1}_B(x) \leq \sup_{B' \in \mathfrak{B}}\|f\|_{\Phi,\phi,B'} \sum_{B \in \mathfrak{B}} \mathbf{1}_B(x) \leq \mathcal{M}_{\mathfrak{B}, \Phi, \phi} f(x) \cdot \sum_{B \in \mathfrak{B}} \mathbf{1}_B(x),
		\end{align}
		and
		\begin{equation}
		\begin{aligned}
			\mathcal{A}^b_{\mathcal{S}, \Phi, \phi} & f(x) = \mathcal{A}^{b, \{1\}, \varnothing}_{\mathcal{S}, \Phi, \phi} f(x) + \mathcal{A}^{b, \varnothing, \{1\}}_{\mathcal{S}, \Phi, \phi} f(x) \\
			& = \sum_{B \in \mathcal{S}} |b(x) - b_B| \, \|f\|_{\Phi, \phi, B} \cdot \mathbf{1}_B(x) + \sum_{B \in \mathcal{S}} \|(b(x) - b_B) f \|_{\Phi, \phi, B} \cdot \mathbf{1}_B(x) \\
			& \leq \sup_{B' \in \mathfrak{B}} |b(x) - b_{B'}| \, \sum_{B \in \mathfrak{B}} \mathbf{1}_B (x) \cdot \mathcal{M}_{\mathfrak{B}, \Phi, \phi} f(x) + \sum_{B \in \mathfrak{B}} \mathbf{1}_B(x) \cdot \mathcal{M}_{\mathfrak{B}, \Phi, \phi} ((b - b_B) f(x)).
		\end{aligned}\label{3.17}
		\end{equation}
		
		Now take a ball \(B \in \mathfrak{B}\) from the family \(\mathfrak{B}\) that satisfies the Besicovitch \(\mathfrak{N}\)-condition. Then from the estimate \eqref{3.16} we obtain
		\begin{align}\label{3.18}
			\|\mathcal{A}_{\mathcal{S}, \Phi, \phi} f\|_{\Psi, \psi, \omega, B} \leq \mathfrak{N} \, \|\mathcal{M}_{\mathfrak{B}, \Phi, \phi} f\|_{\Psi, \psi, \omega, B} \lesssim \mathscr{K} \, \mathfrak{N} \, \|f\|_{\Psi, \psi, \omega, B}.
		\end{align}
		Moreover, combining \eqref{3.17} with Lemma \ref{G-Holder} yields
		\begin{align}\label{3.19}
			\left\|\mathcal{A}^b_{\mathcal{S}, \Phi, \phi} f\right\|_{\Psi, \psi, \omega, B} \leq 2 \, \mathfrak{N} & \, \left\|\sup_{B' \in \mathfrak{B}} |b(x) - b_{B'}|\right\|_{1, \psi, \omega, B} \|1\|_{L^{\infty}} \|\mathcal{M}_{\mathfrak{B}, \Phi, \phi} f\|_{\Psi, \psi, \omega, B} \nonumber \\
			& \hspace{15em} + \mathscr{K} \, \mathfrak{N} \, \|(b - b_B) f \|_{\Psi, \psi, \omega, B}  \nonumber \\
			\lesssim \mathscr{K} \, \mathfrak{N} \, \|b\|_{\mathfrak{L}^1_\psi(\omega)} & \, \|f\|_{\Psi, \psi, \omega, B} + \mathscr{K} \, \mathfrak{N} \, \|b - b_B\|_{1, \psi, \omega, B} \, \|1\|_{L^{\infty}} \, \|f\|_{\Psi, \psi, \omega, B}   \nonumber \\
			\lesssim \mathscr{K} \, \mathfrak{N} \, \|b\|_{\mathfrak{L}^1_\psi(\omega)} & \, \|f\|_{\Psi, \psi, \omega, B}.
		\end{align}
		
		Finally, substituting \eqref{3.18} and \eqref{3.19} into \eqref{ine main1} and \eqref{ine main2} respectively, we obtain \eqref{ine main'1} and \eqref{ine main'2}. This completes the proof of Theorem \ref{main'}.
	\end{proof}

	\section{Dual spaces of abstract Orlicz-Morrey spaces endowed with a ball-basis}\label{Section5}
	
	
	In this section, we aim to discuss the dual space of abstract Orlicz-Morrey spaces endowed with a ball-basis.

	\subsection{Block spaces and their properties}\label{Section5.1}~
	
	Let \((X, \mathfrak{M}, \mu)\) be a measure space. For a Young function \(\Phi\), we define \(\mathscr{M}^{\Phi}(X)\) by
	\begin{align*}
		\mathscr{M}^{\Phi}(X)= \left\{f \in L_{{\rm loc}}^1 (X): \int_X \Phi(k |f(x)|) d\mu(x) < + \infty \text{ for all } k > 0 \right\}.
	\end{align*}
	Clearly, \(\mathscr{M}^{\Phi}(X)\) is a closed subspace of the abstract Orlicz space endowed with a ball-basis (see \cite{Zhou2025}). As in \cite[Page 110, Theorem 7]{Rao1991}, for \(\Phi \in \mathcal{Y}\) we have \((\mathscr{M}^{\Phi}(X))^* = L^{\widetilde{\Phi}}(X)\). Further properties of \(\mathscr{M}^{\Phi}(X)\) can be found in \cite{OM-Nakai-2011, Rao1991}. Before introducing the abstract space \(\mathscr{B}_{(\Phi, \phi)}(X)\) endowed with a ball-basis, we first introduce the notion of a \(\mathfrak{b}(\Phi, \phi)\)-pair.

	\begin{definition}
		Let \((X, \mathfrak{M}, \mu)\) be a measure space endowed with a ball-basis \(\mathfrak{B}\), and let \(\Phi \in \mathscr{Y}\) together with \(\phi \in \mathscr{G}\). A measurable function \(b\) on \(X\) is called a {\tt \((\Phi,\phi)\)-block} if there exists a ball \(B \in \mathfrak{B}\) satisfying the following two conditions:
		\begin{list}{\rm (\theenumi)}{\usecounter{enumi}\leftmargin=1.2cm \labelwidth=1cm \itemsep=0.2cm \topsep=.2cm \renewcommand{\theenumi}{\roman{enumi}}}
			\item\label{5.1-i} \(\supp b \subset B\)
			
			\item\label{5.1-ii} \(b \in \mathscr{M}^{\Phi}(X) \text{ and } \|b\|_{\Phi, \phi, B} \leq \dfrac{1}{\mu(B) \phi(\mu(B))}\).
		\end{list}
		We denote by \(\mathfrak{b}(\Phi, \phi)\) the collection of all pairs \((b, B)\) where \(b\) is a \((\Phi,\phi)\)-block and \(B\) is a ball satisfying \eqref{5.1-i} and \eqref{5.1-ii}.
	\end{definition}
	
	We now present the definition of the abstract space \(\mathscr{B}_{(\Phi, \phi)}(X)\) endowed with a ball-basis.
	
	\begin{definition}
		Let \((X, \mathfrak{M}, \mu)\) be a measure space endowed with a ball-basis \(\mathfrak{B}\), and let \(\Phi \in \mathscr{Y}\) together with \(\phi \in \mathscr{G}\). Given a sequence of \((\Phi,\phi)\)-blocks \(\{b_j\}\) and a sequence of positive numbers \(\{\lambda_j\}\), the abstract space \(\mathscr{B}_{(\Phi, \phi)}(X)\) endowed with a ball-basis is defined by
		\begin{align}\label{B-space}
			\mathscr{B}_{(\Phi, \phi)}(X) = \left\{f : f = \sum_j \lambda_j b_j, \text{ and }\sum_j \lambda_j < \infty \right\}.
		\end{align}
		and its norm is given by
		\begin{align*}
			\|f\|_{\mathscr{B}_{(\Phi, \phi)}(X)} = \inf \left\{ \sum_j \lambda_j : f = \sum_j \lambda_j b_j \in \mathscr{B}_{(\Phi, \phi)}(X) \right\}.
		\end{align*}
	\end{definition}

	We should note that the decomposition \(f = \sum_j \lambda_j b_j\) is not unique. One can readily verify that \(\mathscr{B}_{(\Phi, \phi)}(X)\) is a Banach space and that \(\|f\|_{\mathscr{B}_{(\Phi, \phi)}(X)}\) indeed defines a norm.
	
	Taking \(\omega(x) \equiv 1\) in \eqref{2.2} and setting \(f = 1\), we obtain
	\begin{align}\label{5.2}
		\|1\|_{\Phi, \phi, B} = \frac{1}{\Phi^{-1}(\phi(\mu(B)))}.
	\end{align}
	
	The next lemma reveals the relationship between the space \(\mathscr{B}_{(\Phi, \phi)}(X\) and the spaces \(L^1(X)\) and \(\mathscr{M}^{\Phi}(X)\).

	\begin{lemma}\label{Lm 5.3}
		Let \((X, \mathfrak{M}, \mu)\) be a measure space endowed with a ball-basis \(\mathfrak{B}\). If \((b, B) \in \mathfrak{b}(\Phi, \phi)\), then the following statements hold.
		
		\begin{list}{\rm (\theenumi)}{\usecounter{enumi}\leftmargin=1.2cm \labelwidth=1cm \itemsep=0.2cm \topsep=.2cm \renewcommand{\theenumi}{\arabic{enumi}}}
			\item\label{5.3-1} \(b \in L^1(B)\).
			
			\item\label{5.3-2} There exist constants \(C_1, C_2 > 0\) such that
			\begin{align*}
				\left\{
				\begin{array}{cl}
					\|b\|_{L^1 (B)} \leq C_1, & \text{if } \mu(B) \leq 1, \\
					\|b\|_{L^\Phi (B)} \leq C_2, & \text{if } \mu(B) \geq 1.
				\end{array}
				\right.
			\end{align*}
		\end{list}
	\end{lemma}

	\begin{proof}
		We first verify part \eqref{5.3-1}. Applying H\"older's inequality from \cite[Lemma 2.4]{SZ2026}, we obtain
		\begin{align*}
			\|b\|_{L^1(B)} = \int_B |b| d\mu \leq 2 \mu(B) \phi(\mu(B)) \|b\|_{\Phi, \phi, B} \|1\|_{\widetilde{\Phi}, \phi, B} \leq \frac{2}{\widetilde{\Phi}^{-1}(\phi(\mu(B)))}.
		\end{align*}
		
		Next we prove part \eqref{5.3-2}. If \(\mu(B) \le 1\), note that the function \(\widetilde{\Phi}^{-1}(\phi(r))\) is almost decreasing. Using part \eqref{5.3-1}, we get
		\[
		\|b\|_{L^1(B)} \leq \frac{2}{\widetilde{\Phi}^{-1}(\phi(\mu(B)))} \leq C_1.
		\]
		If \(\mu(B) \ge 1\), we choose \(C_2 \ge 1\) such that \(C_2 \mu(B) \phi(\mu(B)) \ge 1\). Then, since \((b, B) \in \mathfrak{b}(\Phi, \phi)\), we have
		\begin{align*}
			\int_B \Phi \left(\frac{|b(x)|}{C_2}\right) & d\mu(x) \leq \int_B \Phi \left(\frac{|b(x)|}{C_2 \mu(B) \phi(\mu(B)) \|b\|_{\Phi, \phi, B}}\right) d\mu(x) \\
			& \leq \frac{1}{C_2 \mu(B) \phi(\mu(B))} \int_B \Phi \left(\frac{|b(x)|}{\|b\|_{\Phi, \phi, B}}\right) d\mu(x) \leq 1,
		\end{align*}
		which implies \(\|b\|_{L^\Phi(B)} \leq C_2\). This completes the proof of the lemma.
	\end{proof}

	\begin{remark}
		Let \((X, \mathfrak{M}, \mu)\) be a measure space endowed with a ball-basis \(\mathfrak{B}\). Lemma \ref{Lm 5.3} reveals the following two facts:
		\begin{itemize}
			\item The series \(f = \sum_j \lambda_j b_j\) converges in \(L^1(X) \cup \mathscr{M}^{\Phi}(X)\);
			
			\item \(\mathscr{B}_{(\Phi, \phi)}(X) \subset L^1(X) \cup \mathscr{M}^{\Phi}(X)\).
		\end{itemize}
	\end{remark}

	\subsection{Proof of Theorem \ref{main2}}\label{Section5.2}~
	

	In this subsection, we aim to prove Theorem \ref{main2}, which states that the dual space of an abstract Orlicz-Morrey space endowed with a ball-basis is exactly the abstract space \(\mathscr{B}_{(\Phi, \phi)}(X)\) equipped with a ball-basis.

	\begin{proof}[Proof of Theorem \ref{main2}]
		
		{\bf On the one hand}, fix a function \(g \in L^{(\widetilde{\Phi}, \phi)} (X)\). Let \(\{\lambda_j\}\) be a sequence of positive numbers such that \(\sum_j \lambda_j < \infty\), and let \(\{b_j\}\) be a sequence of \((\Phi,\phi)\)-blocks. For any decomposition \(f = \sum_j \lambda_j b_j\) of an element \(f \in \mathscr{B}_{(\Phi,\phi)}(X)\), we consider the fixed \(\lambda_j\) and \(b_j\). Then
		\begin{align*}
			\int_X | b_j g | \, d \mu \leq 2 \mu(B) \phi (\mu(B)) \|b_j\|_{\Phi, \phi, B} \|g\|_{\widetilde{\Phi}, \phi, B} \leq 2 \|g\|_{L^{(\widetilde{\Phi}, \phi)}}.
		\end{align*}
		Consequently,
		\begin{align*}
			\int_X | \sum_j \lambda_j b_j g | \, d\mu \leq \sum_j \lambda_j \int_X | b_j g | \, d\mu \lesssim 1.
		\end{align*}
		Since \(f \in \mathscr{B}_{(\Phi, \phi)}(X) \subset L^1(X) \cup \mathscr{M}^{\Phi}(X)\), we have
		\begin{align*}
			\int_X f g \, d\mu = \int_X \left(\sum_j \lambda_j b_j\right) g \, d\mu = \sum_j \lambda_j \int_X b_j g \, d\mu.
		\end{align*}
		Define a continuous linear functional \(L\) on \(\mathscr{B}_{(\Phi,\phi)}(X)\) by \(L(f) = \int_X f g \, d\mu\). Then
		\begin{align*}
			| L(f)| = \left| \sum_j \lambda_j \int_X b_j g \, d\mu \right| \leq \left( \sum_j \lambda_j \right) \int_X |b_j g| \, d\mu \leq 2 \left( \sum_j \lambda_j \right) \|g\|_{L^{(\widetilde{\Phi}, \phi)}}.
		\end{align*}
		This shows that \(L^{(\widetilde{\Phi}, \phi)} (X) \hookrightarrow \left( \mathscr{B}_{(\Phi, \phi)}(X) \right)^*\).

		{\bf On the other hand}, we prove that \(\left( \mathscr{B}_{(\Phi, \phi)}(X) \right)^* \hookrightarrow L^{(\widetilde{\Phi}, \phi)} (X)\). The proof of this part is divided into the following three steps. To avoid confusion, we denote the above space \(\mathscr{B}_{(\Phi, \phi)} (X)\) by \(\mathscr{B}_{(\Phi, \phi)} (X, \mu)\).
		
		{\it Step 1.} Fix a ball \(B \in \mathfrak{B}\) and define
		\begin{align}\label{5.3}
			d \mu' = \frac{d \mu}{\mu(B) \phi (\mu(B))},
		\end{align}
		that is,
		\begin{align*}
			\mathscr{M}^{\Phi}(X, \mu') = \left\{f \in L_{{\rm loc}}^1 (X): \frac{1}{\mu(B) \phi (\mu(B))} \int_X \Phi(k |f(x)|) d\mu(x) < + \infty \text{ for all } k > 0 \right\},
		\end{align*}
		We observe that if \(f \in \mathscr{M}^{\Phi}(X, \mu')\) and \(\operatorname{supp} f \subset B\), then \(f\) can be written as
		\begin{align*}
			f = \|f\|_{\Phi, \phi, B} \mu(B) \phi(\mu(B)) \cdot \frac{f}{\|f\|_{\Phi, \phi, B} \mu(B) \phi(\mu(B))} =: \|f\|_{\Phi, \phi, B} \mu(B) \phi(\mu(B)) b,
		\end{align*}
		One easily verifies that \(b\) is a \((\Phi,\phi)\)-block. Consequently,
		\begin{align*}
			\|f\|_{\mathscr{B}_{(\Phi, \phi)}(X, \mu)} \leq \|f\|_{\Phi, \phi, B} \mu(B) \phi(\mu(B)),
		\end{align*}
		which shows that \(\mathscr{M}^{\Phi}(B, \mu') \subset \mathscr{B}_{(\Phi, \phi)} (X, \mu)\).

		{\it Step 2.} For the fixed ball \(B \in \mathfrak{B}\) as above, we restrict the linear functional \(L\) on \(\mathscr{B}_{(\Phi, \phi)}(X, \mu)\) to the subspace \(\mathscr{M}^{\Phi}(B, \mu')\). Then, for any \(f \in \mathscr{M}^{\Phi}(B, \mu')\), we define a linear functional \(L'\) on \(\mathscr{M}^{\Phi}(B, \mu')\) by
		\begin{align}\label{5.4}
			L' (f) = \frac{L(f)}{\mu(B) \phi(\mu(B))}.
		\end{align}
		It follows that
		\begin{align*}
			|L' (f)| \leq \frac{\|L\|_{\left( \mathscr{B}_{(\Phi, \phi)}(X) \right)^*} \|f\|_{\mathscr{B}_{(\Phi, \phi)}(X)}}{\mu(B) \phi(\mu(B))} \leq \|L\|_{\left( \mathscr{B}_{(\Phi, \phi)}(X) \right)^*} \|f\|_{\Phi, \phi, B}.
		\end{align*}
		This implies
		\begin{align}\label{5.5}
			\|L'\|_{(\mathscr{M}^{\Phi}(B, \mu'))^*} \leq \|L\|_{\left( \mathscr{B}_{(\Phi, \phi)}(X) \right)^*}.
		\end{align}
		As noted earlier, \((\mathscr{M}^{\Phi} (B, \mu'))^* = L^{\widetilde{\Phi}} (B, \mu')\). Hence there exists a function \(\widetilde{g} \in L^{\widetilde{\Phi}}(B, \mu')\) such that for all \(f \in \mathscr{M}^{\Phi}(B, \mu')\),
		\begin{align}\label{5.6}
			L'(f) = \int f \, \widetilde{g} \, d \mu'.
		\end{align}
		and moreover
		\begin{align}\label{5.7}
			\|\widetilde{g}\|_{\widetilde{\Phi}, \phi, B} = \|\widetilde{g}\|_{L^{\widetilde{\Phi}}(B, \mu')} = \sup_{\|f\|_{\mathscr{M}^{\Phi} (B, \mu')} \leq 1} \left| \int f \, \widetilde{g} \, d\mu' \right| = \sup_{\|f\|_{\mathscr{M}^{\Phi} (B, \mu')} \leq 1} |L'(f)| = \|L'\|_{(\mathscr{M}^{\Phi} (B, \mu'))^*}.
		\end{align}

		{\it Step 3.} For the fixed ball \(B \in \mathfrak{B}\) and any function \(f \in \mathscr{M}^{\Phi}(B, \mu')\), combining \eqref{5.3}, \eqref{5.4} and \eqref{5.6} yields
		\begin{align}\label{5.8}
			L(f) = \int f \, \widetilde{g} \, d \mu.
		\end{align}
		From \eqref{5.5} and \eqref{5.7} we obtain
		\begin{align}\label{5.9}
			\|\widetilde{g}\|_{\widetilde{\Phi}, \phi, B} \leq \|L\|_{\left( \mathscr{B}_{(\Phi, \phi)}(X) \right)^*}.
		\end{align}
		For the above \(\widetilde{g} \in L^{\widetilde{\Phi}}(B, \mu')\), we choose a function \(g \in L^{\Phi}_{\rm loc}(X)\) such that \(g \cdot \mathbf{1}_B = \widetilde{g}\). Then together with \eqref{5.9} we get
		\begin{align}\label{5.10}
			\|g\|_{\widetilde{\Phi}, \phi, B} = \|\widetilde{g}\|_{\widetilde{\Phi}, \phi, B} \leq \|L\|_{\left( \mathscr{B}_{(\Phi, \phi)}(X) \right)^*}.
		\end{align}
		Since the ball \(B\) is arbitrary, \eqref{5.10} implies that \(g \in L^{(\widetilde{\Phi}, \phi)} (X)\) and moreover \(\|g\|_{L^{(\widetilde{\Phi}, \phi)}} \leq \|L\|_{\left( \mathscr{B}_{(\Phi, \phi)}(X) \right)^*}\). For any \(f \in \mathscr{B}_{(\Phi, \phi)}(X, \mu)\) and \(g \in L^{(\widetilde{\Phi}, \phi)}(X)\), we have already noted at the beginning of this theorem that \(f g \in L^1 (X, \mu)\), and the argument above shows that \(L(f) = \int_X f g \, d\mu\) satisfies \eqref{5.10}. This proves the embedding \(\left( \mathscr{B}_{(\Phi, \phi)}(X) \right)^* \hookrightarrow L^{(\widetilde{\Phi}, \phi)} (X)\).		
	\end{proof}

	\begin{remark}
		Theorem \ref{main2} shows that for a measure space \((X, \mathfrak{M}, \mu)\) endowed with a ball-basis \(\mathfrak{B}\), with \(\Phi \in \mathcal{Y}\) and \(\phi \in \mathcal{G}\), the following holds: whenever \(f \in L^{(\widetilde{\Phi}, \phi)} (X)\) and \(g \in \mathscr{B}_{(\Phi, \phi)}(X)\), there exists a constant \(C > 0\) such that
		\begin{align*}
			\int_X |fg| d\mu \leq C \|f\|_{L^{(\widetilde{\Phi}, \phi)} (X)} \|g\|_{\mathscr{B}_{(\Phi, \phi)}(X)}.
		\end{align*}
	\end{remark}

	\section{Appendix}\label{Appendix}
						
	\subsection{Muckenhoupt weights}\label{Appendix1}~
	
	
	Let \((X, \mathfrak{M}, \mu)\) be a measure space endowed with a ball-basis \(\mathfrak{B}\). A measurable function \(\omega\) on \(X\) is called a weight if it satisfies \(0 < \omega(x) < \infty\) for \(\mu\)-almost every \(x \in \bigcup_{B \in \mathfrak{B}} B\).


	We now introduce the {\tt Muckenhoupt class} \(A^\mathfrak{B}_p\). For a given exponent \(1 < p < \infty\), let \(p'\) denote its conjugate exponent. We say \(\omega \in A^\mathfrak{B}_p\) whenever
	\begin{align*}
		[\omega]_{A^\mathfrak{B}_p} := \sup_{B \in \mathfrak{B}} \bigg( \fint_{B} \omega \, d \mu \bigg) \bigg( \fint_{B} \omega^{1-p'} \, d \mu \bigg)^{p-1} < \infty,
	\end{align*} 
	For the endpoint \(p = 1\), we write \(\omega \in A^\mathfrak{B}_1\) if
	\begin{align*}
		[\omega]_{A^\mathfrak{B}_1} := \left\| (M_\mathfrak{B} \omega) \, \omega^{-1} \, \mathbf{1}_{\bigcup_{B \in \mathfrak{B}} B} \right\|_{L^{\infty}(X)} < \infty.
	\end{align*}
	Finally, we treat the case \(p = \infty\) by setting
	\begin{align*}
		A^\mathfrak{B}_\infty = \bigcup_{p \geq 1} A^\mathfrak{B}_p, \quad [\omega]_{A^\mathfrak{B}_\infty} := \inf\{ [\omega]_{ A^\mathfrak{B}_p} : \omega \in A^\mathfrak{B}_p\}.
	\end{align*}

	\subsection{\(\mathscr{Y}\)-class functions}\label{Appendix2}~
	
	A function \(\Phi : [0, +\infty] \to [0, +\infty]\) is said to be a {\tt Young function} provided that the following conditions hold:
	\begin{itemize}
		\item \(\Phi\) is convex.
		\item \(\Phi\) is left-continuous.
		\item \(\displaystyle \lim\limits_{t \to 0^+}\Phi(t) = \Phi(0) = 0\) and \(\displaystyle \lim\limits_{t\to +\infty} \Phi(t) = \Phi(+ \infty) = +\infty\).
	\end{itemize}
	
	Let \(\mathscr{Y}\) denote the collection of all Young functions, that is,
	\begin{align}\label{Y function}
		\mathscr{Y}=\left\{\Phi:\,[0,\,+\infty]\to[0,\,+\infty]\,\big| \, 0 < \Phi(t) < +\infty, \, 0 < t < +\infty \right\}.
	\end{align}
	
	\begin{remark}
		If \(\Phi\in\mathscr{Y}\), then \(\Phi\) is absolutely continuous on any closed subinterval of \([0,\,+\infty)\) and \(\Phi:\,[0,\, +\infty) \to [0,\,+\infty)\) is a bijection.
	\end{remark}
	
	For a Young function \(\Phi\) and for \(0 \leq t \leq +\infty\), we define  
	\begin{align*}
		\Phi^{-1}(t)=\inf\{s \geq 0:\,\Phi(s)>t\} \qquad (\inf \varnothing = +\infty).
	\end{align*}
	
	\begin{remark}
		Concerning the inverse function \(\Phi^{-1}\), the following facts are worth noting.
		\begin{itemize}
			\item \(\Phi\in\mathscr{Y}\) implies that \(\Phi^{-1}\) is the inverse of \(\Phi\) in the usual sense.
			
			\item For every \(0 \leq t < +\infty\), the following double inequality holds:
			\[
			\Phi\left(\Phi^{-1}(t)\right) \leq t \leq \Phi^{-1}\left(\Phi(t)\right).
			\]
		\end{itemize}
	\end{remark}

	For a Young function \(\Phi\), one defines its {\tt complementary function} \(\widetilde{\Phi}\) via the formula
	\begin{align*}
		\widetilde{\Phi}(s)= 
		\left\{ 
		\begin{array}{ll}
			\sup\limits_{t \in [0,\,+\infty)}\{st-\Phi(t)\}, ~~ & s \in [0,\,+\infty), \\
			+\infty,  & s=+\infty. \\
		\end{array}
		\right.
	\end{align*}
	It is clear that \(\widetilde{\Phi}\) is also a Young function and satisfies \(\widetilde{\widetilde{\Phi}}=\Phi\).

	\subsection{\(\mathscr{G}\)-class functions}\label{Appendix3}~
	
	Regarding the concepts of almost increasing (abbreviated as \(a.\,ic.\)), almost decreasing (abbreviated as \(a.\,dc.\)), and doubling condition.
	
	\begin{itemize}
		\item A function \(\varphi:\,(0,\,+\infty) \to (0,\,+\infty)\) is called {\tt almost increasing} (resp. {\tt almost decreasing}) if there exists a constant \(C_{ic} > 0\) (resp. \(C_{dc} > 0\)) such that
		\[
		\varphi(r) \leq C_{ic} \varphi(s) \qquad
		\left(\text{resp. } \varphi(r)\geq C_{dc} \varphi(s) \right) \qquad 
		\text{for } \; r \leq s.
		\]
		
		\item A function \(\varphi:\,(0,\,+\infty) \to(0,\,+\infty)\) is said to satisfy the {\tt doubling condition} if there exists a constant \(C>0\) such that
		\[
		C^{-1} \leq \frac{\varphi(r)}{\varphi(s)} \leq C \qquad 
		\text{for } \; \frac{1}{2} \leq \frac{r}{s} \leq 2.
		\]
		
		\item For functions \(\varphi,\,\vartheta: \,(0,\,+\infty)\to(0,\,+\infty)\) we write \(\varphi\sim\vartheta\) if there exists a constant \(C > 0\) such that
		\[
		C^{-1} \vartheta(r) \leq \varphi(r) \leq C \vartheta(r) \qquad 
		\text{for all } \; r > 0.
		\]
	\end{itemize}
	
	Let \(\mathscr{G}\) denote the set of all functions \(\varphi:\, (0,\, +\infty) \to (0,\,+\infty)\); that is,
	\begin{align}\label{G function}
		\mathscr{G} = \left\{ \varphi:\,\varphi(r) \,\, a.\,dc., \, r\varphi(r) \,\, a.\,ic.\right\}.
	\end{align}
	
	\begin{remark}
		For \(\varphi \in \mathscr{G}\), we have the following assertions.
		\begin{itemize}
			\item[\(\blacksquare\)] \(\varphi\in\mathscr{G}\) implies that \(\varphi\) satisfies the doubling condition.
			
			\item[\(\blacksquare\)] If there exists \(\vartheta:\,(0,\,+\infty)\to(0,\,+\infty)\) and a function \(\varphi \in \mathscr{G}\) such that \(\varphi\sim\vartheta\), then \(\vartheta \in \mathscr{G}\).
		\end{itemize}
	\end{remark}

	\subsection{Proofs of some lemmas}\label{Appendix4}~
	
	
	In this part, we provide some additional proofs omitted from the main text.
	
	\begin{proof}[\bf Proof of Lemma \ref{lemma:AB}]
		
		We first prove part \eqref{AB-1}. Denote \(K = \|f\|_{\Phi, \phi, \omega, B}\) and set the constant \(C = \max\{1, C^{-1}_{dc}\} \max\{1,C_{ic}\} C_{A,B} \geq 1\), noting that \(\Phi\) is a Young function. We then consider two cases.
		
		\begin{itemize}
			\item If \(\omega(A) \leq \omega(B)\), then \( \phi (\omega(A)) \geq C_{dc} \phi (\omega(B))\). Consequently,
			\begin{align*}
				\frac{1}{\omega(A) \phi (\omega(A))} & \int_{A} \Phi \left( \frac{ \left| f \cdot \mathbf{1}_B(x) \right|} {CK} \right) \omega(x) d\mu(x) \leq \frac{1}{\omega(A) \phi(\omega(A))} \cdot \frac{1}{C}\int_B \Phi \left( \frac{|f|}{K} \right) \omega d\mu \\
				& \leq \frac{\omega(B) \phi (\omega(B))}{C \omega(A) \phi (\omega(A))} \leq \frac{\omega(B)} {\omega(A) C_{A,B}} \frac{1}{C_{dc}} \frac{1}{\max\{1,C^{-1}_{dc}\}} \leq 1.
			\end{align*}
			
			\item If \(\omega(B) \leq \omega(A)\), then \(\omega(B) \phi (\omega(B)) \leq C_{ic} \omega(A) \phi(\omega(A))\). Hence,
			\begin{align*}
				\frac{1}{\omega(A) \phi(\omega(A))} & \int_{A} \Phi \left( \frac{\left|f \cdot \mathbf{1}_B(x) \right|} {CK} \right) \omega(x) d\mu(x) \leq \frac{1}{\omega(A) \phi(\omega(A))} \cdot \frac{1}{C} \int_B \Phi \left(\frac{|f|}{K}\right) \omega d\mu \\
				&\leq \frac{\omega(B) \phi(\omega(B))}{C \omega(A) \phi(\omega(A))} \leq \frac{C_{ic}}{\max\{1,C_{ic}\}} \leq 1.
			\end{align*}
		\end{itemize}
		Thus we obtain \(\|f\cdot\mathbf{1}_B\|_{\Phi,\phi,\omega,A} \le C\|f\|_{\Phi,\phi,\omega,B}\), which completes the proof of part \eqref{AB-1}.

		Let \(K' = \|f\|_{\Phi, \phi, \omega, A}\) and let \(K = \|f\|_{\Phi, \phi, \omega, B}\) be as above. Since \(A \subset B\), we obtain
		\[
		\omega(A) \phi(\omega(A)) \leq C_{ic} \omega(B) \phi( \omega(B)) \leq \max\{1, C_{ic}\} \omega(B) \phi(\omega(B)).
		\]
		Set
		\[
		C = \frac{\max\{1,C_{ic}\} \mu(B) \phi(\mu(B))}{\mu(A) \phi(\mu(A))} \geq 1.
		\]
		Then
		\begin{align}\label{ab-1}
			\frac{1}{\omega(B) \phi(\omega(B))} &\int_B \Phi \left( \frac{|f \, \mathbf{1}_A(x)|}{K' \max\{1, C_{ic}\}}\right) \omega(x) d \mu(x) \leq \frac{1}{\omega(B) \phi(\omega(B))} \frac{1}{\max\{1,C_{ic}\}} \int_A \Phi \left( \frac{|f|}{K'} \right) \omega d\mu \nonumber \\
			& \leq \frac{\omega(A) \phi(\omega(A))}{\omega(B) \phi (\omega(B))} \frac{1}{ \max\{1,C_{ic}\}} \leq 1,
		\end{align}
		and
		\begin{align}\label{ab-2}
			\frac{1}{\omega(A) \phi(\omega(A))} \int_{A} \Phi \left( \frac{|f(x)|}{CK} \right) \omega(x) d\mu(x) \leq \frac{1}{C \omega(A) \phi(\omega(A))} \int_{B} \Phi \left( \frac{|f|}{K} \right) \omega d\mu \leq 1.
		\end{align}
		The estimates \eqref{ab-1} and \eqref{ab-2} correspond exactly to part \eqref{AB-2}.
		\end{proof}

	\begin{proof}[\bf Proof of Lemma \ref{PPppw}]
		Fix \(B' \in \mathfrak{B}\). We observe that
		\begin{align}\label{PPppw-1}
			\int_{B'} \Phi & \left( C(\Phi, \Psi)^{-1} \frac{|f(x)|}{\|f\|_{\Psi, \psi, \omega, B'}} \right) \omega(x) \, d \mu(x) \nonumber \\
			& \leq \int_{B'} \Psi \left( \frac{|f(x)|}{\| f \|_{\Psi, \psi, \omega, B'}}\right) \omega(x) \, d \mu(x) \leq \omega(B') \phi (\omega(B')),
		\end{align}
		and
		\begin{align}\label{PPppw-2}
			\int_{B'} \Psi & \left(\frac{|f(x)|}{\max\{1, \, C(\psi, \phi)\} \|f\|_{\Psi, \psi, \omega, B'}} \right) \omega (x) \, d \mu(x)   \nonumber \\
			& \leq \frac{1}{\max\{1, \, C(\psi, \phi)\}} \int_{B'} \Psi \left( \frac{|f(x)|}{\| f \|_{\Psi, \psi, \omega, B'}}\right) \omega(x) \, d \mu(x)    \nonumber \\
			& \leq \frac{\omega(B') \psi (\omega(B'))}{\max\{1, \, C(\psi, \phi)\}} \leq \omega(B') \phi(\omega(B')).
		\end{align}
		The last inequality in \eqref{PPppw-1} and the penultimate inequality in \eqref{PPppw-2} follow from \eqref{remark-2.1}. Since \(B'\) is arbitrary, \eqref{PPppw-1} and \eqref{PPppw-2} imply part \eqref{PPw} and part \eqref{ppw}, respectively.
	\end{proof}

	\subsection{Various definitions of classical Orlicz-Morrey spaces}\label{Appendix5}~
	

	In this part, we mainly supplement the content of the classical Orlicz-Morrey space theory mentioned in Subsection \ref{Section1.1}. We present various types of Orlicz-Morrey spaces that have been studied, together with their historical development and the existing results achieved so far.

	{\bf (1) KK-type Orlicz-Morrey space}\label{Appendix5-1}
	

	The study of Orlicz-Morrey spaces began with Kokilashvili and Krbec \cite{OM-KK-1991}, who in 1991 introduced the first form of Orlicz-Morrey spaces, denoted by \eqref{KK1991} and referred to as KK-type Orlicz-Morrey spaces:
	\begin{align}\label{KK1991}
		\mathfrak{L}^{(\Phi,\lambda)}(\mathbb{R}^n) = \left\{ f \in L^1_{\rm loc}(\mathbb{R}^n) : \|f\|_{\mathfrak{L}^{(\Phi,\lambda)}} < +\infty \right\}.
	\end{align}
	For \(x \in \mathbb{R}^n\), the quantity
	\[
	\|f\|_{\mathfrak{L}^{(\Phi, \lambda)}} = \sup_{r>0} r^{-\lambda} \int_{B(x, r)} \Phi(|f(y)|) \, dy
	\]
	is a quasinorm on \(\mathfrak{L}^{(\Phi,\lambda)}(\mathbb{R}^n)\), where \(0 \le \lambda < n\). In the setting of KK-type Orlicz-Morrey spaces, the boundedness of maximal operators, singular integral operators and their commutators, as well as Littlewood-Paley operators, has been established \cite{OM-KK-1991, OM-KK-1997, OM-KK-2004, OM-KK-2019}.

	{\bf (2) Nakai-type Orlicz-Morrey space}
	

	In 2004, Nakai \cite{OM-Nakai-2004} introduced a second version of Orlicz-Morrey spaces, denoted by \eqref{Nakai2004} and called Nakai-type Orlicz-Morrey spaces:
	\begin{align}\label{Nakai2004}
		L^{(\Phi,\phi)}(\mathbb{R}^n) = \left\{ f \in L^1_{\rm loc}(\mathbb{R}^n) : \|f\|_{L^{(\Phi,\phi)}} < +\infty \right\}.
	\end{align}
	Here the norm is given by
	\[
	\|f\|_{L^{(\Phi,\phi)}} = \sup_{B} \|f\|_{\Phi,\phi,B},
	\]
	where for a ball \(B = B(x,r) \subset \mathbb{R}^n\) with radius \(r\) and volume \(|B|\), one defines
	\[
	\|f\|_{\Phi,\phi,B} = \inf \left\{ \lambda > 0 : \frac{1}{|B|\,\phi(r)} \int_{B} \Phi\!\left(\frac{|f(x)|}{\lambda}\right) dx \le 1 \right\}.
	\]
	In this definition, \(\Phi\) is a Young function and \(\phi \in \mathcal{G}\). Since Nakai-type Orlicz-Morrey spaces have already been discussed in detail in the main body of this paper, we will not repeat the exposition here.

	{\bf (3) SST-type Orlicz-Morrey space}
	
	
	In 2012, Sawano, Sugano, and Tanaka \cite{OM-SST-2012} introduced a third form of Orlicz-Morrey spaces, denoted by \eqref{SST2012} and referred to as SST-type Orlicz-Morrey spaces:
	
	\begin{align}\label{SST2012}
			\mathcal{L}^{\Phi,\varphi}(\mathbb{R}^n) = \left\{ f \in L^1_{\rm loc}(\mathbb{R}^n) : \|f\|_{\mathcal{L}^{\Phi,\varphi}} < +\infty \right\}.
	\end{align}
	The norm is given by
	\[\|f\|_{\mathcal{L}^{\Phi,\varphi}} = \sup_{Q \in \mathscr{Q}} \varphi^{-1}(|Q|)\|f\|_{\Phi,Q}\]
	where
	\[
	\|f\|_{\Phi,Q} = \inf \left\{ \lambda > 0 : \frac{1}{|Q|} \int_{Q} \Phi\!\left(\frac{|f(x)|}{\lambda}\right) dx \le 1 \right\},
	\]
	\(\mathscr{Q}\) denotes the family of all cubes in \(\mathbb{R}^n\) whose sides are parallel to the coordinate axes, and \(|Q|\) is the volume of \(Q\). Here \(\Phi\) is a Young function and \(\phi \in \mathcal{G}\).
		

	The introduction of SST-type Orlicz-Morrey spaces reveals the diversity of this family of function spaces. A natural question then arises: whether there exist connections between different types of such spaces. In 2015, Gala et al. proved that Nakai-type Orlicz-Morrey spaces and SST-type Orlicz-Morrey spaces are isomorphic (see \cite[Theorem 1.6]{OM-SST-2015}). Moreover, many fruitful results have been obtained in the setting of SST-type Orlicz-Morrey spaces. For instance, Hakim et al. \cite{OM-SST-2016} studied fractional maximal operators and certain vector-valued operators on these spaces; Iida \cite{OM-SST-2021} established boundedness results for Orlicz-fractional maximal operators; and Hatano \cite{OM-SST-2025} obtained endpoint estimates for commutators of fractional integral operators.

	{\bf (4) DGS-type Orlicz-Morrey space} 
	
	
	In 2014, Deringoz, Guliyev, and Samko \cite{OM-DGS-2014} introduced a fourth form of Orlicz-Morrey spaces, denoted by \eqref{DGS2014} and called DGS-type Orlicz-Morrey spaces:
	\begin{align}\label{DGS2014}
			\mathcal{M}_{\Phi,\varphi}(\mathbb{R}^n) = \left\{ f \in L^\Phi_{\rm loc}(\mathbb{R}^n) : \|f\|_{\mathcal{M}_{\Phi,\varphi}} < +\infty \right\}.
	\end{align}
	For \(x \in \mathbb{R}^n\), the quantity
	\[
	\|f\|_{\mathcal{M}_{\Phi,\varphi}} = \sup_{r>0} \varphi(x,r)^{-1} \, \Phi^{-1} \left( |B(x,r)|^{-1} \right) \, \|f\|_{L^\Phi(B(x,r))}
	\]
	is a quasi-norm on \(\mathcal{M}_{\Phi,\varphi}(\mathbb{R}^n)\). Here \(L^\Phi\) denotes an Orlicz space, \(\Phi\) is a Young function, and \(\varphi(x,r): \mathbb{R}^n \times (0,\infty) \to (0,\infty)\) is a measurable function. The condition \(f \in L^\Phi_{\rm loc}(\mathbb{R}^n)\) means that for every ball \(B \subset \mathbb{R}^n\), one has \(f \cdot \mathbb{I}_B \in L^\Phi(\mathbb{R}^n)\).


	Among the various types of Orlicz-Morrey spaces, the DGS-type has been the most extensively studied. Boundedness results have been established for many operators on these spaces, including maximal operators \cite{OM-DGS-2014}, fractional maximal operators and their commutators \cite{OM-DGS-2015-2, OM-DGS-2019, OM-DGS-2018}, Calder\'on-Zygmund operators \cite{OM-DGS-2014}, fractional maximal operators and their commutators on spaces of homogeneous type \cite{TSP2020, OM-DGS-2021}, Riesz potentials and their commutators \cite{OM-DGS-2016, OM-DGS-2019}, \(\Phi\)-admissible sublinear integral operators \cite{OM-SST-2014}, as well as intrinsic square operators and their commutators \cite{OM-DGS-2015}.

	{\bf (5) Ho-type Orlicz-Morrey space}
	
	
	In 2023, Ho \cite{OM-H-2023} unified the aforementioned Nakai-type, SST-type, and DGS-type Orlicz-Morrey spaces by introducing a unified formulation, denoted by \eqref{Ho2023} and referred to as Ho-type Orlicz-Morrey spaces:
	\begin{align}\label{Ho2023}
		\mathcal{M}^u_{\Phi,v}(\mathbb{R}^n) = \left\{ f: \|f\|_{\mathcal{M}^u_{\Phi,v}} < +\infty \right\}.
	\end{align}
%
	The norm is given by
	\[\|f\|_{\mathcal{M}^u_{\Phi,v}} = \sup_{B\in\mathcal{B}}\frac{1}{v(B)} \|f\|_{(\Phi,u,B)}\]
	which is a norm on \(\mathcal{M}^u_{\Phi,v}(\mathbb{R}^n)\). In this definition, \(\Phi\) is a Young function, while \(u, v : \mathbb{R}^n \times (0,\infty) \to (0,\infty)\) are Lebesgue-measurable functions. The function \(f\) in \eqref{Ho2023} is a Lebesgue-measurable function, and we write \(v(B(x,r)) = v(x,r)\) for a ball \(B = B(x,r) \in \mathcal{B}\), where \(\mathcal{B}\) denotes the collection of all balls in \(\mathbb{R}^n\). Moreover,
	\begin{align*}
		\|f\|_{(\Phi,u,B)} = \inf\left\{\lambda>0:\,\frac{1}{u(x,r)} \int_{B(x,r)}\Phi\left(\frac{|f(y)}{\lambda}\right)dy\leq 1 \right\}<\infty.
	\end{align*}
	In \cite{OM-H-2023}, Ho analyzed the dual spaces of these Orlicz-Morrey spaces, extended the extrapolation theory to the Orlicz-Morrey setting, and obtained related results such as Agmon-Douglis-Nirenberg estimates for uniformly elliptic equations.

	{\bf (6) Other Orlicz-Morrey spaces}
	

	Depending on the underlying space and whether the measure satisfies a doubling condition, or when the ball \(B\) is taken with its center at the origin under special circumstances, various types of Orlicz-Morrey spaces arise. Examples include central Orlicz-Morrey spaces \cite{COM-2015}, vanishing generalized Orlicz-Morrey spaces \cite{VOM-2014, VOM-2016}, Orlicz-Morrey spaces on spaces of homogeneous type \cite{OM-DGS-2019, TSP2020, OM-DGS-2021}, weighted Orlicz-Morrey spaces \cite{OM-Nakai-2021-2}, and local Orlicz-Morrey spaces \cite{LOM, OM-H-2022}, among others.

	\bigskip
	
	\noindent{\bf\Large Declarations}

	\noindent{\bf Acknowledgements } 
	The authors would like to thank the editors and reviewers for careful reading and valuable comments, which lead to the improvement of this paper.

	\medskip
	\noindent{\bf Mathematics Subject Classification(2020)}
	Primary 42B20; Secondary 42B25; 42B35; 46E30.

	\medskip
	\noindent{\bf Data, materials, and code availability}
	Our manuscript has no associated data, materials, and code.

	\medskip 
	\noindent{\bf Funding information} The research was supported by the National Natural Science Foundation of China (No. 12461021) and the Research Innovation Program for Postgraduates of Xinjiang Uygur Autonomous Region (No. XJ2026G070).

	\medskip
	\noindent{\bf Conflict of interest} The authors declare that they have no conflict of interest.

	\medskip
	\noindent{\bf Author contributions}
	All authors participated in the conception and design of the study and reviewed the manuscript.

	\medskip
	\noindent{\bf Consent for publication}
	The authors agree to the publication.

	\medskip
	\noindent{\bf Ethical statement}
	{\bf (1) Conflict of Interest:} The authors declare that they have no known competing financial interests or personal relationships that could have appeared to influence the work reported in this paper. {\bf (2) Data Availability:} No new data were generated or analyzed in support of this research. {\bf (3) Authorship and Originality:} This manuscript is the authors' original work and has not been published or submitted elsewhere. All authors have read and approved the final version and agree to be accountable for all aspects of the work.

\end{document}